\documentclass[reqno,11pt]{amsart}
\usepackage{amsmath,amssymb,amsthm,mathrsfs}
\usepackage{epsfig,color}
\usepackage[latin1]{inputenc}
\usepackage[english]{babel}
\usepackage[numbers]{natbib}
\usepackage[colorlinks, citecolor=blue, linkcolor=red]{hyperref}

\numberwithin{equation}{section}

\newtheorem{ackn}{Acknowledgments\!}

\def\00{{\bf 0}}

\def\RR{\mathbb R}

\newcommand{\diver}{{\rm div \,}}

\newcommand{\eps}{{\varepsilon}}

\newtheorem*{theorem*}{Theorem}
\newtheorem{theorem}{Theorem}[section]
\newtheorem{lemma}[theorem]{Lemma}
\newtheorem{proposition}[theorem]{Proposition}

\newtheorem{corollary}[theorem]{Corollary}
\newtheorem{remark}[theorem]{Remark}

\newtheorem{definition}[theorem]{Definition}

\begin{document}
    \title[On the critical p-Laplace equation]{On the critical $p$-Laplace equation}

  \date{}

\author{Giovanni Catino, Dario D. Monticelli, Alberto Roncoroni}

\address{G. Catino (corresponding author), Dipartimento di Matematica, Politecnico di Milano, Piazza Leonardo da Vinci 32, 20133, Milano, Italy.}
\email{giovanni.catino@polimi.it}

\address{D. Monticelli, Dipartimento di Matematica, Politecnico di Milano, Piazza Leonardo da Vinci 32, 20133, Milano, Italy.}
\email{dario.monticelli@polimi.it}

\address{A. Roncoroni, Dipartimento di Matematica, Politecnico di Milano, Piazza Leonardo da Vinci 32, 20133, Milano, Italy.}
\email{alberto.roncoroni@polimi.it}

\begin{abstract}
In this paper we provide the classification of positive solutions to the critical $p-$Laplace equation on $\mathbb{R}^n$, for  $1<p<n$, possibly having infinite energy. If $n=2$, or if $n=3$ and $\frac 32<p<2$ we prove rigidity without any further assumptions. In the remaining cases we obtain the classification under energy growth conditions or suitable control of the solutions at infinity. Our assumptions are much weaker than those already appearing in the literature. We also discuss the extension of the results to the Riemannian setting.
\end{abstract}

\maketitle

\begin{center}


\medskip

\centerline{\bf AMS subject classification: 35J92, 35B33, 35B06, 58J05, 53C21}

\end{center}

\

\section{Introduction}

Given $n\geq 2$ and $1<p<n$ we consider positive solutions of the well-known critical $p-$Laplace equation
\begin{equation}\label{eq-pl}
\Delta_p u + u^{p^*-1} = 0 \quad \text{ in } \mathbb{R}^n\, ,
\end{equation}
where $\Delta_p$ is the usual $p-$Laplace operator and $p^*$ is the Sobolev exponent,  explicitly
$$
\Delta_p u:=\mathrm{div}(\vert\nabla u\vert^{p-2}\nabla u) \,, \quad \text{and} \quad p^*=\frac{np}{n-p}\, .
$$
The critical $p-$Laplace equation has been the object of several studies in the differential geometry and in the PDE's communities, indeed problem \eqref{eq-pl} is related to the study of the critical points of the Sobolev inequality (see e.g.  the survey \cite{Ron_lincei}) and, for $p=2$, to the Yamabe problem (see e.g. the survey \cite{LP}).  An interesting and challenging problem is the classifications of solutions to \eqref{eq-pl}: one can show that the following functions
\begin{equation} \label{tal}
\mathcal U_{\lambda,x_0} (x) := \left( \frac{\lambda^{\frac{1}{p-1}}\left(n^{\frac{1}{p}} \left(\frac{n-p}{p-1}\right)^{\frac{p-1}{p}} \right)}{\lambda^\frac{p}{p-1} + |x-x_0|^\frac{p}{p-1} } \right)^{\frac{n-p}{p}} \,, \quad \lambda >0 \,, \ x_0 \in \mathbb{R}^n \,.
\end{equation}
form a $2-$parameters family of solutions to \eqref{eq-pl}\footnote{if one chooses a different normalization constant in the numerator of the functions $\mathcal U_{\lambda,x_0}$, then one obtains a constant $k\neq 1$ as a coefficient of $u^{p^\ast-1}$  in equation \eqref{eq-pl}.}. Hence the natural question is the following:

\

{\em Given a positive solution to \eqref{eq-pl}, is it of the form \eqref{tal}?}

\

\noindent The functions described in \eqref{tal} are usually called Aubin-Talenti bubbles, since in two independent papers Aubin, in \cite{Aubin}, and Talenti, in \cite{Tal}, prove that the functions \eqref{tal} realize the equality in the sharp Sobolev inequality in $\RR^n$.

It is well known (see e.g. \cite{ding} and \cite{clap}) that there exist multiple sign-changing, non-radial, finite energy solutions to
$$
\Delta_p u + u|u|^{p^*-2} = 0 \quad \text{ in } \mathbb{R}^n.
$$
In this paper we focus on non-negative weak solutions to \eqref{eq-pl} which, by the maximum principle for quasilinear equations (see e.g. \cite{vaz}), are either zero or positive.

Turning back to the classification results of positive solutions to \eqref{eq-pl} in the seminal paper \cite{CGS} (see also \cite{Obata} and \cite{GNN} for previous important results) the authors consider the semilinear case (i.e. $p=2$) and they prove that positive smooth solutions to \eqref{eq-pl} with $p=2$ are given by the Aubin-Talenti bubbles \eqref{tal} (see also \cite{CL} and \cite{LZ}).  The proof is based on a refinement of the method of moving planes (introduced in \cite{Alex} in the context of constant mean curvature hypersurfaces and transposed in \cite{Serrin} and in \cite{GNN_bis} to study of qualitative properties of solutions of the PDE's) and on the Kelvin transform.

The quasilinear case (i.e. $1<p<n$, $p\neq 2$) is more complicated since,  for example, the Kelvin transform is not available. Nevertheless the method of moving planes has been exploited in \cite{DMMS,Sc,Vetois} to prove the following classification result:

\

{\em Let $u$ be a positive weak solution of equation \eqref{eq-pl} with finite energy, i.e.
$$
u\in\mathcal{D}^{1,p}(\mathbb{R}^n):=\lbrace u\in L^{p^{\ast}}(\mathbb{R}^n)\, : \, \nabla u\in L^p(\mathbb{R}^n)\rbrace \, ,
$$

then $u(x)=\mathcal U_{\lambda,x_0} (x)$ for some $\lambda>0$ and $x_0\in\RR^n$.}

\

We mention that this result has been recently generalized in \cite{CFR} in the anisotropic setting (see also \cite{ERSV}) and in convex cones of $\mathbb{R}^n$ (see also \cite{LPT}) and in \cite{catmon,FMM,muso} in the Riemannian setting (see Appendix \ref{Riem} for a more detailed discussion).

\

As far as we know trying to prove the same result without the assumption that $u$ has finite energy is an open and challenging problem for $p\neq 2$; in this paper we deal with this problem obtaining a classification  result of all positive weak solutions of \eqref{eq-pl} in dimensions $n=2,3$ for $\frac{n}{2}<p<2$, while for different values of $n$ and $p$ we require that $u$ satisfies suitable conditions at infinity (which are weaker than the finite energy assumption).

\

Before stating our results we recall the variational nature of the critical $p-$Laplace equation \eqref{eq-pl}: the energy associated to \eqref{eq-pl} is given by
$$
E_{\mathbb{R}^n}(u):=\frac 1p\int_{\mathbb{R}^n} |\nabla u|^p + \frac{1}{p^*}\int_{\mathbb{R}^n} u^{p^*}\, ,
$$
indeed it is well-known that the Euler-Lagrange equation associated to this energy functional is \eqref{eq-pl}.  We define the energy on a general open set $\Omega\subseteq\RR^n$ and we split it as follows:
$$
E_\Omega(u)=E^{\text{kin}}_\Omega(u)+E^{\text{pot}}_\Omega(u):=\frac 1p\int_\Omega |\nabla u|^p + \frac{1}{p^*}\int_\Omega u^{p^*}.
$$
\

With these notations our first theorem is a rigidity result under a growth assumption of the energy on annuli, indeed we have the following:

\begin{theorem}\label{t-e1} Let $u$ be a positive weak solution of equation \eqref{eq-pl}.
If one of the following holds
\begin{itemize}
\item[(i)] $1<p\leq \frac{2n}{n+1}$ and
$$
E_{A_R}(u)=O\left(R^{\frac{n}{n-1}}\right),\quad\text{or}
$$

\item[(ii)] $\frac{2n}{n+1}<p<2$ and
$$
E_{A_R}(u)=O\left(R^{\frac{(2-p)(n-p)}{2(p-1)^2}}\right),\quad\text{or}
$$
\item[(iii)] $p>2$ and $
u(x)\leq C|x|^{\alpha}$ as $|x|\to\infty$ for some $\alpha\geq 0$ and
$$
E_{A_R}(u)=O\left(R^k\right)\quad\text{for some}\quad
k<\frac{2(n-p)}{2+(n-3)p}-\alpha\frac{p[n(p-2)+p]}{(n-p)[2+(n-3)p]},
$$
\end{itemize}
where $A_R:=B_{2R}\setminus B_{R}$, then $u(x)=\mathcal U_{\lambda,x_0} (x)$ for some $\lambda>0$ and $x_0\in\RR^n$.
\end{theorem}

In particular this result implies the classification of solutions $u\in \mathcal{D}^{1,p}(\mathbb{R}^n)$, simply by observing that
$$
u\in\mathcal{D}^{1,p}(\mathbb{R}^n) \quad \Longleftrightarrow \quad E_{\RR^n}(u)<\infty\, .
$$
Here one also has to recall that positive solutions $u\in \mathcal{D}^{1,p}(\mathbb{R}^n)$ have the following behavior:
$$
u(x)\leq \frac{C}{1+|x|^{\frac{n-p}{p-1}}}\qquad\text{and}\qquad |\nabla u(x)|\leq \frac{C}{1+|x|^{\frac{n-1}{p-1}}}\, ,
$$
for all $x\in\RR^n$ and some $C>0$, as it was shown in \cite[Theorem 1.1]{Vetois}. We explicitly note that we are not using these estimates in our proofs .

Note also that, if $u$ is a positive weak solution of equation \eqref{eq-pl}, by Lemmas \ref{l-dario} and \ref{l-2} for every $\alpha>0$ one has
$$
E_{A_R}^{\text{pot}}(u)=O(R^\alpha) \quad \Longleftrightarrow\quad E_{A_R}(u)=O(R^\alpha) \quad \Longleftrightarrow\quad E_{A_R}^{\text{kin}}(u)=O(R^\alpha).
$$
See Remark \ref{rem-bo} for the proof.

\

For suitable choices of $n$ and $p$ we can show rigidity results without any further assumptions on the solution.

\begin{theorem}\label{t-se0} Let $u$ be a positive weak solution of equation \eqref{eq-pl}. If one of the following holds
\begin{itemize}
\item[(i)] $n=2$ and $1<p<2$, or
\item[(ii)] $n=3$ and $\frac 32<p<2$,
\end{itemize}
then $u(x)=\mathcal U_{\lambda,x_0} (x)$ for some $\lambda>0$ and $x_0\in\RR^n$.
\end{theorem}

In our last two classification theorems, where we consider general $n$ and $p$, we rely on conditions on the behavior of the solution at infinity, which are much weaker than all the results already available in the literature. For $1<p<2$ we have he following:

\begin{theorem}\label{t-se1} Let $u$ be a positive weak solution of equation \eqref{eq-pl} with
$$
u(x)\leq C|x|^{\alpha}\qquad\text{as } |x|\to\infty,
$$
for some
$$\alpha<\bar{\alpha}:=\tfrac{(3p-n)(n-p)}{p(n-2p)}.$$
If one of the following holds
\begin{itemize}
\item[(i)] $n=3$ and $1<p\leq \frac32$, or
\item[(ii)] $n\geq 4$ and $1<p<2$,
\end{itemize}
then $u(x)=\mathcal U_{\lambda,x_0} (x)$ for some $\lambda>0$ and $x_0\in\RR^n$.
\end{theorem}

For $2<p<n$, on the other hand, we have:

\begin{theorem}\label{t-se2} Let $u$ be a positive weak solution of equation \eqref{eq-pl} with
$$
u(x)\leq C|x|^{\alpha}\qquad\text{as } |x|\to\infty.
$$
Let
$$
\hat{\alpha}:=\tfrac{2(n-p)}{p(p-2)},\qquad\check{\alpha}:=\tfrac{(n-p)^2}{(p-2)(p-1)},\qquad \bar{\alpha}:=\tfrac{(3p-n)(n-p)}{p(n-2p)}, \qquad \tilde{\alpha}:=\tfrac{(3p-n)(n-p)}{p(n-3p+2)}.
$$
Assume that one of the following holds
\begin{itemize}
\item[(i)] $n=3$, and $2<p<3$ and  $\alpha<\check{\alpha}$;
\item[(ii)] $n=4$, and  $$2<p<\check{p}\quad\text{and}\quad\alpha<\hat{\alpha},$$ or $$\check{p}\leq p <4\quad\text{and}\quad\alpha<\check{\alpha};$$
\item[(iii)] $n=5$ or $n=6$, and $$2<p<\tfrac{n+2}{3}\quad\text{and}\quad\alpha<\bar{\alpha},$$ or $$\tfrac{n+2}{3}\leq p <\check{p}\quad\text{and}\quad\alpha<\hat{\alpha},$$ or
$$\check{p}\leq p <n\quad\text{and}\quad\alpha<\check{\alpha};$$
\item[(iv)] $n\geq 7$ and $$2<p\leq\tfrac{n}{3}\quad\text{and}\quad\alpha<\tilde{\alpha},$$ or $$\tfrac{n}{3}<p<\tfrac{n+2}{3}\quad\text{and}\quad\alpha<\bar{\alpha},$$ or $$\tfrac{n+2}{3}\leq p <\check{p}\quad\text{and}\quad\alpha<\hat{\alpha},$$ or
$$\check{p}\leq p <n\quad\text{and}\quad\alpha<\check{\alpha},$$
\end{itemize}
where $\check{p}=\tfrac{n-2+\sqrt{n^2-4n+12}}{2}$. Then $u(x)=\mathcal U_{\lambda,x_0} (x)$ for some $\lambda>0$ and $x_0\in\RR^n$.
\end{theorem}

In particular we have the following classifications:

\begin{corollary}\label{c-1.5} Let $u$ be a positive bounded weak solution of equation \eqref{eq-pl} on $\RR^n$, with $n\leq 6$, or $n\geq 7$ and $p>\frac n3$.  Then $u(x)=\mathcal U_{\lambda,x_0} (x)$ for some $\lambda>0$ and $x_0\in\RR^n$.
\end{corollary}

\begin{corollary} Let $u$ be a positive weak solution of equation \eqref{eq-pl} on $\RR^n$ with
$$u(x)\leq C|x|^{-\frac{n-p}{p}}.$$ Then $u(x)=\mathcal U_{\lambda,x_0} (x)$ for some $\lambda>0$ and $x_0\in\RR^n$.
\end{corollary}

We stress the fact that all the limiting exponents in Theorems \ref{t-se1} and \ref{t-se2} are strictly larger than $-\frac{n-p}{p}$ in the ranges of $n$ and $p$ where they are used. Actually they are strictly positive under the hypothesis of Corollary \ref{c-1.5}. Note that the exponent $-\frac{n-p}{p}$ is the threshold decay in order for a radial solution to have finite energy.

\

As an auxiliary result, which may have an independent interest, we prove a gradient estimate for positive solutions to \eqref{eq-pl} which is instrumental in the proofs of Theorems \ref{t-e1} and \ref{t-se2} in the case $2<p<n$.
\begin{proposition}\label{p-ge} Let $u$ be a positive weak solution of equation \eqref{eq-pl} with $1<p<n$. Then, for every $0<\eps<\frac{p-1}{n-p}$ it holds
$$
|\nabla u|\leq C\left(\sup_{B_{2R}(x_0)} u^{\frac{1}{n-p}+\eps}+R^{-\eps\frac{n-p}{p-1}}\right) u^{\frac{n-1}{n-p}-\eps}\qquad\text{on }B_{R}(x_0)
$$
for some $C=C(n,p,\eps)>0$, for every $R>0$ and every $x_0\in\RR^n$.
\end{proposition}
This estimate is sharp for the positive solutions $\mathcal U_{\lambda,x_0} (x)$.

\

Most of the available classification results for the critical $p-$Laplace equation are based on a careful application of the moving plane technique. Interesting exceptions can be found in \cite{catmon, CFR, FMM} where the authors, exploiting integral estimates obtained through test functions arguments, prove the classification via the vanishing of a suitable traceless tensor field depending on the solutions and their derivatives. Similar estimates have been used by Gidas and Spruck \cite{GS} and Serrin and Zou \cite{serzou} in the subcritical case. In this paper we adopt a similar approach; the starting point in the proof of our classification results is the key integral estimate in Corollary \ref{c-fond} which in turn is obtained adapting arguments in \cite{serzou} to the critical case.

\

One of the nice features of our approach is that it can be quite easily extended to the Riemannian setting, as it was shown in the case $p=2$ in \cite{catmon}. We review all the needed steps in order to adapt our arguments to the case of a Riemannian manifold $(M^n,g)$ in the Appendix \ref{Riem}, where we sketch the proof of the following:

\begin{theorem} Let $u$ be a positive weak solution with regularity \eqref{eq-r1}--\eqref{eq-r4} on a complete non-compact Riemannian manifold $(M^n,g)$ such that
\begin{itemize}
\item[(i)] $\mathrm{Ric}\geq 0$, if $1<p<2$,  or
\item[(ii)] $\mathrm{Sec}\geq 0$, if $2<p<n$.
\end{itemize}
Then, under the hypotheses of Theorems \ref{t-e1}-\ref{t-se0}-\ref{t-se1}-\ref{t-se2},  $(M^n,g)$ is isometric to $\RR^n$ with the Euclidean metric and $u(x)=\mathcal U_{\lambda,x_0} (x)$ for some $\lambda>0$ and $x_0\in\RR^n$.
\end{theorem}

\

The paper is organized as follows: in Section \ref{Pre} we collect some useful preliminary results that will be needed in the proof of our main theorems, in particular in Corollary \ref{c-fond} we prove a key integral estimate that will be the starting point in the proofs of the main results; in Section \ref{gr-est} we prove the sharp gradient estimate in Proposition \ref{p-ge}; in Section \ref{s-ec} we give the proof of Theorem \ref{t-e1}; in Section \ref{s-ci} we show Theorems \ref{t-se1} and \ref{t-se2}. Finally, in Appendix \ref{Riem} we discuss the generalizations to the Riemannian setting.

\

\section{Preliminaries}\label{Pre}

\subsection{The key integral estimate} In this part we collect some well-known facts about equation \eqref{eq-pl}: the definition of weak solutions and of sub/super-solutions to \eqref{eq-pl} and  the regularity theory related to the $p-$Laplace equation. Moreover we show the main integral estimate that we are going to use to prove our rigidity results.

\begin{definition}\label{deb}
  A weak solution of \eqref{eq-pl} is a function $u\in W^{1,p}_{\text{loc}}(\RR^n)\cap L^{\infty}_{\text{loc}}(\RR^n)$ such that
  $$
  \int_{\mathbb{R}^n}|\nabla u|^{p-2}\langle\nabla u,\nabla\psi\rangle-\int_{\mathbb{R}^n}u^{p^*-1}\psi=0\qquad \forall \psi\in W^{1,p}_0(\RR^n).
  $$
  where $W^{1,p}_0(\RR^n)$ denotes the set of compactly supported functions of $W^{1,p}(\RR^n)$.

Moreover,  a weak subsolution of \eqref{eq-pl} is a function $u\in W^{1,p}_{\text{loc}}(\RR^n)\cap L^{\infty}_{\text{loc}}(\RR^n)$ such that
  $$
  \int_{\mathbb{R}^n}|\nabla u|^{p-2}\langle\nabla u,\nabla\psi\rangle-\int_{\mathbb{R}^n}u^{p^*-1}\psi\leq 0\qquad \forall \psi \in W^{1,p}_0(\RR^n)\,,
  $$
  such that $\psi$ is non-negative.  Finally,  $u\in W^{1,p}_{\text{loc}}(\RR^n)\cap L^{\infty}_{\text{loc}}(\RR^n)$ is a weak supersolution of \eqref{eq-pl} if the opposite inequality holds.
\end{definition}
Thanks to the regularity theory in \cite{ACF} (see Theorems 1.1 and 1.4) we have that any weak solution of \eqref{eq-pl} satisfies
\begin{equation}\label{eq-r3}
u\in W^{2,2}_{\text{loc}}(\RR^n\setminus\Omega_{cr})\cap C^{1,\alpha}_{\text{loc}}(\RR^n),
\end{equation}
for some $\alpha\in(0,1)$ and, in addition,
\begin{equation}\label{eq-r2}
\vert\nabla u\vert^{p-2} \nabla u\in W^{1,2}_{\text{loc}}(\RR^n),
\end{equation}
and 
\begin{equation}\label{eq-r4}
\vert\nabla u\vert^{p-2} \nabla^2 u\in L^{2}_{\text{loc}}(\RR^n\setminus\Omega_{cr}),
\end{equation}
for all $1<p<n$, where
$$
\Omega_{cr}=\{x\in\RR^n\,|\,\nabla u(x)=0\}.
$$
We note that by a bootstrap argument, any weak solution is actually $C^{\infty}$ on $\Omega_{cr}^c$. Since $u>0$ in $\RR^n$ it is easy to see that $\Omega_{cr}$ has zero measure (see e.g. \cite{ACF}). In particular (see also equation $(10)$ in \cite{ACF})
\begin{equation}\label{eq-r5}
\vert\nabla u\vert^{p-2} \nabla^2 u\in L^{2}_{\text{loc}}(\RR^n).
\end{equation}
Moreover,  if $1<p\leq 2$ then 
\begin{equation}\label{eq-r1}
u\in W^{2,2}_{\text{loc}}(\RR^n)\cap C^{1,\alpha}_{\text{loc}}(\RR^n),
\end{equation}
for some $\alpha\in(0,1)$. 

\

Here, also in order to help the reader, we adopt the same notation as in \cite[Chapter II]{serzou}. If $u$ is a positive solution of \eqref{eq-pl} we define the vector fields
\begin{equation}\label{vectorfield1}
  \mathbf{u}:=|\nabla u|^{p-2}\nabla u,\qquad\mathbf{v}:=u^{-\frac{n(p-1)}{n-p}}|\nabla u|^{p-2}\nabla u.
\end{equation}
It is then convenient to redefine $\nabla\mathbf{u}$ and $\nabla\mathbf{v}$ to be $0$ on $\Omega_{cr}$, for $1<p<n$. Then we set
\begin{equation}\label{vectorfield3}
\mathbf{U}:=\begin{cases}
             \nabla \mathbf{u}, & \mbox{in }\Omega_{cr}^c  \\
             0, & \mbox{in }\Omega_{cr},
           \end{cases}
           \qquad
\mathbf{V}:=\begin{cases}
             \nabla \mathbf{v}, & \mbox{in }\Omega_{cr}^c  \\
             0, & \mbox{in }\Omega_{cr}.
           \end{cases}
\end{equation}
Moreover, we recall the definition of the \textit{traceless} tensor
\begin{equation*}
\mathring{\mathbf{V}}:=\mathbf{V}-\frac{\operatorname{tr}\mathbf{V}}{n} \operatorname{Id}_n
\end{equation*}
where $\operatorname{Id}_n$ is the identity tensor.

\

Using this notation, we have the following fundamental estimate.

\begin{proposition}\label{p-mest}
  Let $1<p<n$, $p\neq2$, and $u$ be a positive weak solution of equation \eqref{eq-pl}. Then for every $0\leq \phi\in C^\infty_0(\mathbb{R}^n)$ we have
  $$
  \int u^{\frac{(n-1)p}{n-p}}|\mathring{\mathbf{V}}|^2\phi\leq-\int u^{\frac{(n-1)p}{n-p}}\langle \mathbf{v}\cdot\mathring{\mathbf{V}},\nabla\phi\rangle
  $$
\end{proposition}
\begin{remark}
We remark that the result holds with the equality sign when $p>2$. It also holds when $p=2$ with the equality sign, replacing $\mathbf{V}$ with $\nabla\mathbf{v}$ (see Proposition 6.2 in \cite{serzou}).
\end{remark}
\begin{proof}
  In case $p>2$ the result follows from Proposition 6.2 in \cite{serzou}. When $1<p<2$, due to problems of regularity, one can first use Proposition 7.1 in \cite{serzou}, where a truncation of $|\nabla u|$ is introduced in order to deal with the critical set of $u$, and then pass to the limit as the relevant parameter $\varepsilon$ tends to $0$ to conclude. We provide here some details to help the reader.

We start with the case $p>2$.   Formula (6.16) in Proposition 6.2 in \cite{serzou} with 
\begin{equation}\label{1}
a=-\frac{n(p-1)}{n-p},\qquad b=\frac{p(n-1)}{n-p},\qquad q=\frac{np}{n-p}-1
\end{equation}
reads as
$$
\int (u^b I+\psi)\phi=-\int \langle\mathbf{\omega},\nabla\phi\rangle.
$$
Here
\begin{equation*}
\mathbf{\omega}:=u^{\frac{(n-1)p}{n-p}}\left(\mathbf{v}\cdot\mathbf{V}-\frac{1}{n}\mathbf{v}\operatorname{tr} \mathbf{V}\right)=u^{\frac{(n-1)p}{n-p}}\mathbf{v}\cdot\mathring{\mathbf{V}}.
\end{equation*}
The expression $\mathbf{v}\cdot\mathbf{V}$ is interpreted as the vector with components $(\mathbf{v}\cdot \mathbf{V})_i=\mathbf{v}_j\mathbf{V}_{ij}$, for $i=1,\ldots,n$, where we use the Einstein convention of summation over repeated indices. Moreover
  $$
  I:=|\mathbf{V}|^2-\frac{1}{n}\operatorname{tr}(\mathbf{V})^2\equiv|\mathring{\mathbf{V}}|^2
  $$
  and
  $$
  \psi:=u^{b+2a+q-1}\left(A+q\hat{A}\right)|\nabla u|^p+Bu^{b+2a-2}|\nabla u|^{2p}+C\operatorname{div}\left(u^{b+2a-1}|\nabla u|^p\mathbf{u}\right)\equiv0,
$$
since by our choices of $a,b,q$ one easily computes $\left(A+q\hat{A}\right)=B=C=0$, using their explicit expressions in \cite{serzou}. Some comments are in order: thanks to the regularity of $u$ stated in \eqref{eq-r2} and \eqref{eq-r5} we have
\begin{equation}\label{bor}
\mathbf{u}\in W^{1,2}_{loc}(\RR^n),\quad u^b \mathbf{v}\in W^{1,2}_{loc}(\RR^n),\quad \diver{\mathbf{v}}\in W^{1,2}_{loc}(\RR^n).
\end{equation}
In particular Lemma 6.4 and Lemma 6.5 in \cite{serzou} apply and formula (6.17) in \cite{serzou} holds. The rest of the proof of Proposition 6.2 in \cite{serzou} goes through using \eqref{bor}. Thus the result immediately follows, with the equality sign.

\medskip

  The case $1<p<2$ is more involved. For any fixed $\varepsilon\in(0,1)$, following \cite{serzou} we set
  $$
  |\nabla u|_\eps=\max\{|\nabla u|,\eps\},\qquad \mathbf{u}_\eps=|\nabla u|_\eps^{p-2}\nabla u,\qquad \mathbf{v}_\eps=u^{-\frac{n(p-1)}{n-p}}\mathbf{u}_\eps,\qquad
  \mathbf{V}_\eps=\nabla \mathbf{v}_\eps.
  $$
We observe that \eqref{bor} holds also in this case, therefore the proof of Proposition 7.1 in \cite{serzou} goes through and we obtain
  \begin{equation}\label{2}
  \int (u^b I_\eps+\psi_\eps)\phi=-\int \langle\mathbf{\omega}_\eps,\nabla\phi\rangle+O(\eps^{2(p-1)}).
  \end{equation}
  with
 \begin{align*}
 I_\eps&=\operatorname{tr}(\mathbf{V}_\eps\mathbf{V})-\frac{1}{n}\operatorname{tr}\mathbf{V}_\eps\operatorname{tr}\mathbf{V}\\
 \omega_\eps&=\left(\mathbf{v}_\eps\cdot\mathbf{V}-\frac{1}{n}\mathbf{v}_\eps\operatorname{tr} \mathbf{V}\right)u^{\frac{(n-1)p}{n-p}}=\mathbf{v}_\eps\cdot\mathring{\mathbf{V}}u^{\frac{(n-1)p}{n-p}}\\
 \psi_\eps&=u^{b+2a+q-1}\left(\bar{A}+q\hat{A}\right)\Gamma_\eps+\tilde{A}u^{b+2a-1}|\nabla u|^p\diver\mathbf{u}_\eps\\
 &\quad+Bu^{b+2a-2}|\nabla u|^{p}\Gamma_\eps+C\operatorname{div}\left(u^{b+2a-1}|\nabla u|^p\mathbf{u}_\eps\right),
\end{align*}
  where
  $$
  \Gamma_\eps=\langle\mathbf{u}_\eps,\nabla u\rangle=|\nabla u|_\eps^{p-2}|\nabla u|^2,
  $$
  see formulas (7.6) and (6.15) in \cite{serzou}. Using their explicit expressions provided in \cite{serzou}, one can easily see that choosing $a,b,q$ as in \eqref{1} we get $B=C=\bar{A}+\tilde{A}+q\hat{A}=0$.

  Since $u\in C^1$ and positive, $\mathbf{u}_\eps$ and $\mathbf{v}_\eps$ converge to $\mathbf{u}$ and $\mathbf{v}$ in $L^2_{\text{loc}}$ respectively. Since $\mathbf{V}$ and $\mathring{\mathbf{V}}$ are in $L^2_{\text{loc}}$ and  $u\in C^1$ and positive, $\mathbf{\omega}_\eps$ converges to $\mathbf{\omega}$ in $L^1_{\text{loc}}$ as $\eps$ tends to $0$.

  Moreover, $\Gamma_\eps$ converges to $|\nabla u|^p$ in $L^2_{\text{loc}}$ and $|\nabla u|^p\diver\mathbf{u}_\eps$ converges weakly in $L^2_{\text{loc}}$ to $$|\nabla u|^p\diver\mathbf{u}=-|\nabla u|^pu^{q},$$ see the proof of Proposition 7.2 in \cite{serzou}, since for $1<p< 2$ we have $u\in W^{2,2}_{loc}(\RR^n)$. Then we obtain that $\psi_\eps$ converges weakly to $0$ in $L^2_{\text{loc}}$ as $\eps$ tends to $0$.

  Finally we consider the term $I_\eps$. Let $$\Omega_\eps=\{x\in\RR^n\,|\,0<|\nabla u|<\eps\},$$ then $I_\eps=I\geq0$ almost everywhere on $\Omega_\eps^c$ and $I_\eps=I=0$ on $\Omega_{cr}$. Thus
  $$
  \int \phi u^b I_\eps = \int_{\Omega_\eps^c}\phi u^b I+\int_{\Omega_\eps} \phi u^b I_\eps.
  $$
  By formula (7.22) in \cite{serzou} we have
  $$
  \liminf_{\eps\rightarrow0}\int_{\Omega_\eps} \phi u^b I_\eps\geq0,
  $$
  while by the monotone convergence theorem we conclude that
  $$
  \lim_{\eps\rightarrow0}\int_{\Omega_\eps^c}\phi u^b I=\int \phi u^b I.
  $$
  Passing to the limit as $\eps$ tends to $0$ in \eqref{2} we have
$$
\int u^b I\phi\leq-\int \langle\mathbf{\omega},\nabla\phi\rangle,
$$
which is the desired inequality.
  \end{proof}

An easy consequence of the previous proposition is the following key integral estimate, from which we will deduce our main rigidity results.

\begin{corollary}\label{c-fond}
  Let $1<p<n$, $p\neq2$, and $u$ be a positive weak solution of equation \eqref{eq-pl}. Then for every $0\leq \eta\in C^\infty_0(\mathbb{R}^n)$ and $l\geq 2$ we have
$$
  \int u^{\frac{(n-1)p}{n-p}}|\mathring{\mathbf{V}}|^2\eta^l\leq C\int u^{\frac{(2-p)n-p}{n-p}}|\nabla u|^{2(p-1)} |\nabla  \eta|^2\eta^{l-2}
$$
and
$$
  \int u^{\frac{(n-1)p}{n-p}}|\mathring{\mathbf{V}}|^2\eta^l\leq C\left( \int_{\text{supp}|\nabla \eta|} u^{\frac{(n-1)p}{n-p}}|\mathring{\mathbf{V}}|^2\eta^l\right)^{\frac12}\left(\int u^{\frac{(2-p)n-p}{n-p}}|\nabla u|^{2(p-1)} |\nabla  \eta|^2\eta^{l-2}\right)^{\frac12}
$$
\end{corollary}
\begin{proof} We consider  $\phi=\eta^l$ in Proposition \ref{p-mest}. Using Cauchy-Schwarz and Young's inequalities and the definition of $\mathbf{v}$ we get
\begin{align*}
 \int u^{\frac{(n-1)p}{n-p}}|\mathring{\mathbf{V}}|^2\eta^l &\leq l \int u |\mathring{\mathbf{V}}| |\nabla u|^{p-1}|\nabla \eta| \eta^{l-1}\\
 &\leq \frac12 \int u^{\frac{(n-1)p}{n-p}}|\mathring{\mathbf{V}}|^2\eta^l + C\int u^{\frac{(2-p)n-p}{n-p}}|\nabla u|^{2(p-1)} |\nabla  \eta|^2\eta^{l-2}
\end{align*}
and the first inequality follows. The second part of the statement can be obtained similarly, using H\"{o}lder's inequality instead.
\end{proof}
\begin{remark} A similar estimate appears in \cite{catmon} in the case $p=2$.
\end{remark}

\

\subsection{Some a priori estimates} We collect here some general lemmas concerning the behavior of positive solutions  of the equation \eqref{eq-pl}, that we will need in the proofs of our main theorems.
The first is a lower bound for positive $p$-superharmonic functions.
\begin{lemma}[{\cite[Lemma 2.3]{serzou}}]\label{l-ser} Let $u$ be a positive weak solution of
$$
\Delta_p u \leq 0
$$
on $\RR^n\setminus K$ with $K$ compact and $1<p<n$. Then there exist positive constants $\rho,A>0$ such that
$$
u(x) \geq \frac{A}{|x|^{\frac{n-p}{p-1}}}\quad\text{for all } x\in B_\rho^c.
$$
\end{lemma}

The next lemma provides bounds for the kinetic energy in terms of the potential energy for positive weak subsolutions of equation \eqref{eq-pl}. In particular, solutions $u\in L^{p*}(\RR^n)$ automatically have finite energy.

\begin{lemma}\label{l-dario} Let $u$ be a positive weak solution of
$$
-\Delta_p u \leq u^{p^*-1} \qquad \text{ in } \mathbb{R}^n.
$$
Then, for every $\eps>0$ there exists a constant $C=C(n,p)>0$ such that for every $R>0$
\begin{align*}
\int_{B_{2R}\setminus B_R}  |\nabla u|^p  &\leq C\left(1+\eps^{-\frac{p}{n-p}}\right)\int_{B_{5R/2}\setminus B_{R/2}} u^{p^*} + C\eps,
\end{align*}
Moreover, there exists $C=C(n,p)>0$ such that for every $R>0$
$$
\int_{B_R}|\nabla u|^p \leq C\int_{B_{2R}} u^{p^*} +C\left(\int_{B_{2R}} u^{p^*} \right)^{\frac{n-p}{n}}.
$$
In particular, if $u\in L^{p*}(\RR^n)$, then $|\nabla u|\in L^p (\RR^n)$, i.e. $E_{\RR^n}(u)<\infty$.
\end{lemma}
\begin{remark} A estimate similar to the second part of the statement appears in the proof of Theorems 1.2 and 1.4 in \cite{catmon} in the case $p=2$.
\end{remark}

\begin{proof}
Testing the weak formulation given in Definition \ref{deb} with $u \eta^q$, with $q>1$ and where $\eta\in C^\infty_0(\mathbb{R}^n)$, we obtain
\begin{equation*}
\int u^{\frac{np}{n-p}}\eta^q \geq  \int  |\nabla u|^p \eta^q + q\int u|\nabla u|^{p-2}\langle \nabla u, \nabla \eta\rangle\eta^{q-1} \, .
\end{equation*}
i.e., from Cauchy-Schwarz and Young inequalities
\begin{align}\label{eqbo}\nonumber
 \int  |\nabla u|^p \eta^q &\leq \int u^{\frac{np}{n-p}}\eta^q   + q\int u|\nabla u|^{p-1}\vert\nabla \eta\vert\eta^{q-1} \\
 &\leq \int u^{\frac{np}{n-p}}\eta^q + \frac12 \int \vert\nabla u\vert^{p} \eta^{q} + C\int \vert\nabla \eta\vert^p u^p \eta^{q-p}\\\nonumber
 &\leq  C\left(1+\eps^{-\frac{p}{n-p}}\right) \int u^{\frac{np}{n-p}}\eta^q + \frac12 \int \vert\nabla u\vert^{p} \eta^{q} + \eps \int \vert\nabla \eta\vert^n \eta^{q-n},
\end{align}
for every $\eps>0$. Let $q>n$, for any $R>1$, we choose $\eta\in C^{\infty}_0(\RR^n)$  such that $\eta\equiv 1$ in $B_{2R}\setminus B_R$, $\eta\equiv 0$ in $B_{R/2}\cup B_{5R/2}^c$, $0\leq \eta\leq 1$ on $\RR^n$ and $\eta$ satisfies
$$
|\nabla \eta|^{2} \leq C R^{-2}\quad\text{in }(B_{5R/2}\setminus B_{2R})\cup(B_{R}\setminus B_{R/2}).
$$
Hence, for every $\eps>0$, we get
\begin{align*}
\int_{B_{2R}\setminus B_R}  |\nabla u|^p  &\leq C\left(1+\eps^{-\frac{p}{n-p}}\right)\int_{B_{5R/2}\setminus B_{R/2}} u^{\frac{np}{n-p}} + C\eps,
\end{align*}
which is the first part of the statement. In order to prove the second part, let $q>p$ and for any $R>1$ choose $\eta\in C^{\infty}_0(\RR^n)$ such that $\eta\equiv 1$ in $B_R$, $\eta\equiv 0$ in $B_{2R}^c$, $0\leq \eta\leq 1$ on $\RR^n$ and $\eta$ satisfies
$$
|\nabla \eta|^{2} \leq C R^{-2} \quad\text{in }B_{2R}\setminus B_{R}.
$$
From \eqref{eqbo} using H\"older inequality we get
\begin{align*}
\int_{B_{R}}  |\nabla u|^p  &\leq C\int_{B_{2R}} u^{\frac{np}{n-p}} +  \dfrac{C}{R^p} \int_{B_{2R}\setminus B_{R}}  u^p  \\
 & \leq C\int_{B_{2R}} u^{\frac{np}{n-p}} +  \dfrac{C}{R^p}\left( \int_{B_{2R}}  u^{\frac{np}{n-p}} \right)^{\frac{n-p}{n}} \vert B_{2R}\setminus B_{R}\vert^{\frac{p}{n}}\\
 & \leq C\int_{B_{2R}} u^{\frac{np}{n-p}} +  C\left( \int_{B_{2R}}  u^{\frac{np}{n-p}} \right)^{\frac{n-p}{n}}\, .
\end{align*}
\end{proof}

\

Similar to the previous lemma, the following provides bounds for the potential energy in terms of the kinetic energy for positive weak supersolutions of the equation \eqref{eq-pl}. In particular, solutions with $\nabla u\in L^{p}(\RR^n)$ automatically have finite energy. Since the proof is similar to the previous one, we will omit it.

\begin{lemma}\label{l-2}
Let $u$ be a positive weak solution of
$$
-\Delta_p u \geq u^{p^*-1} \qquad \text{ in } \mathbb{R}^n.
$$
Then, for every $\eps>0$ there exists a constant $C=C(n,p)>0$ such that for every $R>0$
\begin{align*}
\int_{B_{2R}\setminus B_R} u^{p^*}   &\leq C\left(1+\eps^{-\frac{p}{n(p-1)}}\right)\int_{B_{5R/2}\setminus B_{R/2}} |\nabla u|^p + C\eps,
\end{align*}
Moreover, there exists $C=C(n,p)>0$ such that for every $R>0$
$$
\int_{B_R}u^{p^*} \leq C\int_{B_{2R}} |\nabla u|^{p} +C\left(\int_{B_{2R}} |\nabla u|^{p} \right)^{\frac{n(p-1)}{n(p-1)+p}}.
$$
In particular, if $\nabla u\in L^{p}(\RR^n)$, then $u\in L^{p^*} (\RR^n)$, i.e. $E_{\RR^n}(u)<\infty$.
\end{lemma}

\begin{remark}\label{rem-bo}
As already observed in the introduction, if $u$ is a positive weak solution of equation \eqref{eq-pl}, by Lemmas \ref{l-dario} and \ref{l-2} for every $\alpha>0$ one has
$$
E_{A_R}^{\text{pot}}(u)=O(R^\alpha) \qquad \Longleftrightarrow\qquad E_{A_R}(u)=O(R^\alpha) \qquad \Longleftrightarrow\qquad E_{A_R}^{\text{kin}}(u)=O(R^\alpha).
$$
Indeed by Lemma \ref{l-dario} there exists $C>0$ such that, for every $R>0$, we have
\begin{align*}
  E_{A_R}^{\text{pot}}(u) & \leq E_{A_R}(u)\leq (C+1)\int_{B_{5R/2}\setminus B_{R/2}}u^{p^*}+1 \\
  &\leq(C+1)\left(E_{A_{R/2}}^{\text{pot}}(u)+E_{A_R}^{\text{pot}}(u)+E_{A_{2R}}^{\text{pot}}(u)\right)+1.
\end{align*}
Similarly by Lemma \ref{l-2} for every $R>0$
\begin{align*}
  E_{A_R}^{\text{kin}}(u) & \leq E_{A_R}(u)\leq (C+1)\int_{B_{5R/2}\setminus B_{R/2}}|\nabla u|^{p}+1 \\
  &\leq(C+1)\left(E_{A_{R/2}}^{\text{kin}}(u)+E_{A_R}^{\text{kin}}(u)+E_{A_{2R}}^{\text{kin}}(u)\right)+1.
\end{align*}
Thus we conclude.
\end{remark}

\

\section{A sharp gradient estimate}\label{gr-est}

In this section we will prove the sharp gradient estimate in Proposition \ref{p-ge}. To the best of our knowledge this result is new and we believe that it may have independent interest.

We begin by defining the (second order part of the) linearized $p$-Laplace operator (see e.g \cite{kotni, val})
$$
P_f(w):=|\nabla f|^{p-2}\Delta w+(p-2)|\nabla f|^{p-4}\nabla^2 w (\nabla f,\nabla f).
$$
Observe that $P_f(f)=\Delta_p f$. The following inequality follows from the extension to $p$-Laplace operator of the classical Bochner formula (see e.g. \cite{kotni, val}).
\begin{lemma}\label{pboch} Given $x\in\RR^n$, a domain $U$ containing $x$ and a function $f\in C^3(U)$, if $|\nabla f|(x)\neq 0$, at $x$ it holds
\begin{align*}
\frac1p P_f(|\nabla f|^p) &\geq \frac1n (\Delta_p f)^2 +\frac{n}{n-1}\left(\frac1n \Delta_p f - (p-1)|\nabla f|^{p-4}\nabla^2 f (\nabla f, \nabla f)\right)^2\\
&\quad+ |\nabla f|^{p-2}\left[\langle \nabla f, \nabla \Delta_p f\rangle - (p-2) \frac{\Delta_p f}{|\nabla f|^2}\nabla^2 f (\nabla f, \nabla f)\right]
\end{align*}
\end{lemma}
\begin{proof} It follows combining the $p$-Bochner formula \cite[Proposition 3.1]{val} with the sharp estimate \cite[Lemma 3.2]{val}.
\end{proof}
\subsection{Proof of Proposition \ref{p-ge}} Let $u$ be a positive weak solution of equation \eqref{eq-pl} with $1<p<n$ and define $f:=u^a>0$, $a\in\RR\setminus\{0\}$. By regularity we know that $f$ is smooth where $|\nabla f|>0$. Thus, where $|\nabla f|>0$, we have
$$
\nabla f = a u^{a-1}\nabla u
$$
and,
\begin{align*}
\Delta_p f &=\diver(|\nabla f|^{p-2}\nabla f) \\
&=a|a|^{p-2}\diver(u^{(a-1)(p-1)}|\nabla u|^{p-2}\nabla u)\\
&=a|a|^{p-2}u^{(a-1)(p-1)}\Delta_p u + a|a|^{p-2}(a-1)(p-1)u^{(a-1)(p-1)-1}|\nabla u|^p\\
&=-a|a|^{p-2}u^{(a-1)(p-1)+q}+ a|a|^{p-2}(a-1)(p-1)u^{(a-1)(p-1)-1}|\nabla u|^p\\
&=-a|a|^{p-2}f^{\frac{(a-1)(p-1)+q}{a}}+ \frac{(a-1)(p-1)}{a}\frac{|\nabla f|^p}{f}\,,
\end{align*}
with $q=p^*-1$. Using Lemma \ref{pboch}, where $|\nabla f|>0$, we obtain
\begin{align*}
\frac1p P_f(|\nabla f|^p) &\geq \tfrac1n (\Delta_p f)^2 +\tfrac{n}{n-1}\left(\frac1n \Delta_p f - (p-1)|\nabla f|^{p-4}\nabla^2 f (\nabla f, \nabla f)\right)^2\\
&\quad+ |\nabla f|^{p-2}\left[\langle \nabla f, \nabla \Delta_p f\rangle - (p-2) \frac{\Delta_p f}{|\nabla f|^2}\nabla^2 f (\nabla f, \nabla f)\right] \\
&\geq\tfrac{1}{n-1}(\Delta_p f)^2 +|\nabla f|^{p-2}\langle \nabla f, \nabla \Delta_p f\rangle-\tfrac{2(p-1)+(n-1)(p-2)}{(n-1)}|\nabla f|^{p-4}\Delta_p f\nabla^2 f (\nabla f, \nabla f)\\
&=\tfrac{1}{n-1}(\Delta_p f)^2 +|\nabla f|^{p-2}\langle \nabla f, \nabla \Delta_p f\rangle+c_1|\nabla f|^{-2}\Delta_p f\langle\nabla|\nabla f|^p,\nabla f\rangle\\
&\geq\tfrac{(1-a)(p-1)[p-1+a(n-p)]}{a^2(n-1)}f^{-2}|\nabla f|^{2p}-c_2f^{\frac{(a-1)(p-1)+q}{a}-1}|\nabla f|^p\\
&\quad+\left(c_1|\nabla f|^{-2}\Delta_p f+c_3f^{-1}|\nabla f|^{p-2}\right)\langle\nabla|\nabla f|^p,\nabla f\rangle\\
&=\tfrac{(1-a)(p-1)[p-1+a(n-p)]}{a^2(n-1)}f^{-2}|\nabla f|^{2p}-c_2f^{\frac{(a-1)(p-1)+q}{a}-1}|\nabla f|^p\\
&\quad+\left(c_4f^{\frac{(a-1)(p-1)+q}{a}}|\nabla f|^{-2}+c_5f^{-1}|\nabla f|^{p-2}\right)\langle\nabla|\nabla f|^p,\nabla f\rangle
\end{align*}
with
$$
c_1=-\tfrac{2(p-1)+(n-1)(p-2)}{p(n-1)},\qquad c_2=|a|^{p-2}\left[\tfrac{n+1}{n-1}(a-1)(p-1)+q\right],
$$
$$
c_3=\tfrac{(a-1)(p-1)}{a},\qquad c_4=-a|a|^{p-2}c_1,\qquad c_5=\tfrac{(a-1)(p-1)}{a}c_1+c_3.
$$
We choose $a$
\begin{equation}\label{a}
a:=-\frac{p-1}{n-p}+\eps
\end{equation}
for a given $0<\eps<\frac{p-1}{n-p}$. Then there exists $\lambda=\lambda(\eps)>0$ such that, where $|\nabla f|>0$, we have
\begin{align}\label{estp}
\frac1p P_f(|\nabla f|^p) &\geq \lambda f^{-2}|\nabla f|^{2p}-c_2f^{\frac{(a-1)(p-1)+q}{a}-1}|\nabla f|^p\\\nonumber
&\quad+\left(c_4f^{\frac{(a-1)(p-1)+q}{a}}|\nabla f|^{-2}+c_5f^{-1}|\nabla f|^{p-2}\right)\langle\nabla|\nabla f|^p,\nabla f\rangle.
\end{align}
If $|\nabla f|$ achieves its maximum in $B_{2R}$ at some point $\bar{x}\in B_{2R}$, we have
$$
\nabla |\nabla f|^p=0\qquad\text{and}\qquad P_f(|\nabla f|^p)\leq 0\qquad\text{at } \bar{x}
$$
and \eqref{estp} implies, at $\bar{x}$,
$$
0\geq f^{-2}|\nabla f|^p\left(\lambda|\nabla f|^p -c_2f^{\frac{(a-1)(p-1)+q}{a}+1}\right)=f^{-2}|\nabla f|^p\left(\lambda|\nabla f|^p -c_2u^{(a-1)(p-1)+q+a}\right).
$$
Moreover, since $q+1=\frac{np}{n-p}$ we have
\begin{equation}\label{teta}
\theta:=(a-1)(p-1)+q+a=\frac{p}{n-p}+p\eps>0.
\end{equation}
We obtain
$$
\sup_{B_{2R}}|\nabla f| \leq C\sup_{B_{2R}} u^{\frac{\theta}{p}}\quad\Longleftrightarrow\quad |\nabla u(x)|\leq C\left(\sup_{B_{2R}} u^{\frac{1}{n-p}+\eps}\right) u(x)^{\frac{n-1}{n-p}-\eps}
$$
for all $x\in B_{2R}$.

On the other hand, if $|\nabla f|$ does not achieve its maximum at some point $\bar{x}$, we have to employ a cutoff argument. For a given $0<\delta<\frac{1}{2}$ there exist nonnegative cutoff functions  $\phi=\phi(|x|)$ with $\phi\equiv 1$ on $B_R$, $\phi\equiv 0$ on $B_{2R}^c$, $0\leq\phi\leq 1$ on $\RR^n$ and such that
\begin{equation}\label{cut-off}
|\nabla \phi|\leq \frac{C}{R} \phi^{1-\delta},\qquad |\nabla^2 \phi|\leq \frac{C}{R^2}\phi^{1-2\delta},
\end{equation}
on $B_{2R}\setminus B_R$ for some $C>0$\footnote{We observe that such cutoff functions can be obtained setting $\phi(x)=\psi\left(\frac{\vert x\vert}{R}\right)^{1/\delta}$,  where $\psi\in C^2([0,\infty))$ is such that $\psi\equiv 1$ in $[0,1)$, $\psi\equiv 0$ in $[2,\infty)$ and $0\leq\psi\leq 1$.}. Let
$$
H:=\phi|\nabla f|^p
$$
and $\bar{x}$ be a maximum point of $H$. We can assume that $\phi(\bar{x})>0$ and $|\nabla f|(\bar{x})>0$. At $\bar{x}$ we have
$$
\nabla H=0,\qquad P_{f} H\leq 0,
$$
Therefore, at $\bar{x}$, we have
\begin{equation}\label{nhz}
\nabla H=0\quad\Longleftrightarrow\quad \nabla |\nabla f|^p = -\phi^{-2}H\nabla \phi.
\end{equation}
Moreover, using \eqref{nhz}, we have
$$
\nabla_i\nabla_j H = \phi^{-1}H\nabla_i\nabla_j\phi-2\phi^{-2}H\nabla_i\phi\nabla_j\phi+\phi\nabla_i\nabla_j|\nabla f|^p
$$
and thus
$$
\Delta H =  \phi^{-1}H\Delta\phi-2\phi^{-2}H|\nabla\phi|^2+\phi\Delta|\nabla f|^p.
$$
Using the definition of $P_f$, at $\bar{x}$ we obtain
\begin{align*}
0 &\geq |\nabla f|^{4-p}P_f H \\
&= |\nabla f|^2 \phi^{-2}\left(\phi\Delta\phi-2|\nabla\phi|^2\right)H+\phi|\nabla f|^2\Delta|\nabla f|^p\\
&\quad+(p-2)\phi^{-2}\left[\phi\nabla^2\phi(\nabla f,\nabla f)H-2\langle \nabla\phi,\nabla f\rangle^2H+\phi^3\nabla^2|\nabla f|^p(\nabla f,\nabla f)\right]\\
&\geq \phi|\nabla f|^{4-p}P_f|\nabla f|^p-\frac{C}{R^2}\phi^{-2\delta}|\nabla f|^2 H,
\end{align*}
i.e.
\begin{equation}\label{eq11}
0\geq \phi^{1+2\delta}P_f|\nabla f|^p-\frac{C}{R^2}\phi^{-\frac{p-2}{p}}H^{\frac{2(p-1)}{p}}.
\end{equation}
From \eqref{estp}, we have
\begin{align*}
\frac1p P_f(|\nabla f|^p) &\geq \lambda f^{-2}\phi^{-2}H^2-c_2f^{-2}u^{\theta}\phi^{-1}H-\left(c_4f^{-1}u^{\theta}|\nabla f|^{-2}+c_5f^{-1}|\nabla f|^{p-2}\right)\phi^{-2}\langle\nabla\phi,\nabla f\rangle H,
\end{align*}
where $\theta>0$ is defined in \eqref{teta}. We get
\begin{align*}
\frac1p P_f(|\nabla f|^p) \geq \lambda f^{-2}\phi^{-2}H^2-Cf^{-2}u^{\theta}\phi^{-1}H-\frac{C}{R}f^{-1}u^{\theta}\phi^{-1-\delta+\frac 1p}H^{\frac{p-1}{p}}-\frac{C}{R}f^{-1}\phi^{-1-\delta-\frac{p-1}{p}}H^{\frac{2p-1}{p}}.
\end{align*}
Using Lemma \ref{l-ser}, on $B_{2R}$ we have
$$
f = u^a \leq C R^{-a\frac{n-p}{p-1}}=C R^{1-\eps\frac{n-p}{p-1}}\quad\Longleftrightarrow\quad -\frac{1}{R}\geq-\frac{CR^{-\eps\frac{n-p}{p-1}}}{f}
$$
Therefore
\begin{align*}
\frac1p P_f(|\nabla f|^p) \geq f^{-2}\left(\lambda \phi^{-2}H^2-C\phi^{-1}u^{\theta}H-C\phi^{-1-\delta+\frac 1p}u^{\theta}R^{-\eps\frac{n-p}{p-1}}H^{\frac{p-1}{p}}-C\phi^{-1-\delta-\frac{p-1}{p}}R^{-\eps\frac{n-p}{p-1}}H^{\frac{2p-1}{p}}\right).
\end{align*}
Using \eqref{eq11} we obtain
\begin{align*}
0 &\geq \lambda \phi^{-2}H^2-C\phi^{-1}u^{\theta} H-C\phi^{-1-\delta+\frac 1p}u^{\theta}R^{-\eps\frac{n-p}{p-1}}H^{\frac{p-1}{p}}\\
&\quad-C\phi^{-1-\delta-\frac{p-1}{p}}R^{-\eps\frac{n-p}{p-1}}H^{\frac{2p-1}{p}}-C\phi^{-1-2\delta-\frac{p-2}{p}}R^{-2\eps\frac{n-p}{p-1}}H^{\frac{2(p-1)}{p}}.
\end{align*}
Now, choosing $\delta<\min\lbrace\frac 1p,\frac 12\rbrace$, since $0\leq\phi\leq 1$, at $\bar{x}$ we obtain
\begin{align*}
0&\geq  \lambda H^2-Cu^{\theta}H-Cu^{\theta}R^{-\eps\frac{n-p}{p-1}}H^{\frac{p-1}{p}}-CR^{-\eps\frac{n-p}{p-1}}H^{\frac{2p-1}{p}}-CR^{-2\eps\frac{n-p}{p-1}}H^{\frac{2(p-1)}{p}}\\
&\geq \frac{\lambda}{2} H^2-C\left(u^{2\theta}+R^{-2\eps\frac{p(n-p)}{p-1}}\right),
\end{align*}
where we used Young's inequality\footnote{We remark that on the third term $u^{\theta}R^{-\eps\frac{n-p}{p-1}}H^{\frac{p-1}{p}}$ we used the following generalization of the classical Young's inequality:
$$
abc\leq \varepsilon a^r+ k_1(\varepsilon)b^s + k_2(\varepsilon)c^t\, ,
$$
for all $a,b,c\geq 0$, $\varepsilon>0$ and where $r,s,t>1$ are such that $\frac{1}{r}+\frac{1}{s}+\frac{1}{t}=1$.}. This clearly implies
$$
H\leq H(\bar{x})\leq C\left(u^{\theta}(\bar{x})+R^{-\eps\frac{p(n-p)}{p-1}}\right)\qquad\text{on }B_{2R}
$$
for every $R>0$, and in particular
$$
|\nabla f|^p\leq C\left(\sup_{B_{2R}}u^{\theta}+R^{-\eps\frac{p(n-p)}{p-1}}\right)\qquad\text{on }B_{R},
$$
i.e.
$$
|\nabla u|\leq C\left(\sup_{B_{2R}} u^{\frac{1}{n-p}+\eps}+R^{-\eps\frac{n-p}{p-1}}\right) u^{\frac{n-1}{n-p}-\eps}\qquad\text{on }B_{R}
$$
which is the thesis.
\begin{flushright}
$\square$
\end{flushright}

\

In case $u$ is controlled by a power of the distance at infinity, one can obtain the following point-wise estimate on the gradient of $u$ in terms of $u$. In \cite[Corollary 2.2]{FMM} the authors obtained a logarithmic gradient estimate in case $p=2$.

\begin{corollary}\label{c-ge}
Let $u$ be a positive weak solution of equation \eqref{eq-pl} with $1<p<n$. Assume
$$
u(x)\leq C |x|^\alpha
$$
for all $x\in B_1^c$, for some $\alpha\in\RR$. Then, for every $0<\eps<\frac{p-1}{n-p}$ it holds
$$
|\nabla u(x)|\leq C\left(|x|^{\left(\frac{1}{n-p}+\eps\right)\alpha}+|x|^{-\eps\frac{n-p}{p-1}}\right) u(x)^{\frac{n-1}{n-p}-\eps}
$$
for some $C=C(n,p,\eps,\alpha)>0$, for every $x\in B_4^c$.
\end{corollary}
\begin{proof} It is sufficient to apply Proposition \ref{p-ge} on $B_{|x|/4}(x)$ at the point $x$.
\end{proof}

\

\section{Rigidity with energy control}\label{s-ec}

In this section we prove Theorem \ref{t-e1}.

\subsection{Proof of Theorem \ref{t-e1} (i)} Let $u$ be a positive weak solution of equation \eqref{eq-pl} with $1<p\leq \frac{2n}{n+1}$. From Corollary \ref{c-fond} with $l=2$ we have
\begin{equation}\label{est}
  \int u^{\frac{(n-1)p}{n-p}}|\mathring{\mathbf{V}}|^2\eta^2\leq C\int u^{\frac{(2-p)n-p}{n-p}}|\nabla u|^{2(p-1)} |\nabla  \eta|^2.
\end{equation}
for every $\eta\in C^\infty_0(\mathbb{R}^n)$ with $\eta\geq0$. For any $R>1$, we choose $\eta\in C^{\infty}_0(\RR^n)$ such that $\eta\equiv 1$ in $B_R$, $\eta\equiv 0$ in $B_{2R}^c$, $0\leq \eta\leq 1$ on $\RR^n$ and $\eta$ satisfies
$$
|\nabla \eta|^{2} \leq C R^{-2} \quad\text{in }A_R=B_{2R}\setminus B_{R}.
$$
We show that the integral on the righthand side of \eqref{est} is uniformly bounded in $R$. Since $p\leq \frac{2n}{n+1}<2$, we have
$$
\frac{(2-p)n-p}{n-p}\geq 0\qquad\text{and}\qquad 2(p-1)<p.
$$
If $1<p<\frac{2n}{n+1}$, using H\"older inequality we obtain
\begin{align*}
\int u^{\frac{(2-p)n-p}{n-p}}|\nabla u|^{2(p-1)} |\nabla  \eta|^2 &\leq \frac{C}{R^2}\left(\int_{A_R}u^{p^*}\right)^{\frac{(2-p)n-p}{np}}\left(\int_{A_R}|\nabla u|^p\right)^{\frac{2(p-1)}{p}}|A_R|^{\frac 1n}\\
&\leq \frac{C}{R} E^{\text{pot}}_{A_R}(u)^{\frac{(2-p)n-p}{np}}E^{\text{kin}}_{A_R}(u)^{\frac{2(p-1)}{p}}\\
&\leq \frac{C}{R} E_{A_R}(u)^{\frac{n-1}{n}}.
\end{align*}

If $p=\frac{2n}{n+1}$, similarly we have
\begin{align*}
\int u^{\frac{(2-p)n-p}{n-p}}|\nabla u|^{2(p-1)} |\nabla  \eta|^2 &=\int |\nabla u|^{\frac{2(n-1)}{n+1}} |\nabla  \eta|^2\\
&\leq\frac{C}{R^2}\left(\int_{A_R}|\nabla u|^{\frac{2n}{n+1}}\right)^{\frac{n-1}{n}}|A_R|^{\frac 1n}\\
&\leq \frac{C}{R} E^{\text{kin}}_{A_R}(u)^{\frac{n-1}{n}}\\
&\leq \frac{C}{R} E_{A_R}(u)^{\frac{n-1}{n}}.
\end{align*}

Thanks to the energy assumptions, in both cases we have  that the righthand side of \eqref{est} is uniformly bounded in $R$. Hence
$$
\int_{\RR^n} u^{\frac{(n-1)p}{n-p}}|\mathring{\mathbf{V}}|^2<\infty,
$$
and by the second inequality in Corollary \ref{c-fond}, passing to the limit as $R$ tends to infinity, we obtain
$$
\int_{\RR^n} u^{\frac{(n-1)p}{n-p}}|\mathring{\mathbf{V}}|^2=0,
$$
i.e.
\begin{equation}\label{eq12}
\mathring{\mathbf{V}} = \nabla  \mathbf{v}-\frac{\diver\mathbf{v}}{n}\operatorname{Id}_n \equiv 0\qquad\text{in } \Omega_{cr}^c.
\end{equation}
Let $\Omega_0\subseteq\Omega_{cr}^c $ be a connected component of $\Omega_{cr}^c$. Since $0<u\in C^{1,\alpha}_{\text{loc}}(\RR^n)$ then
$$
v=u^{-\frac{p}{n-p}}\in C^{1,\alpha}_{\text{loc}}(\RR^n).
$$
Since
$$\mathbf{v}=-\left(\frac{n-p}{p}\right)^{p-1}|\nabla v|^{p-2}\nabla v$$
we get
\begin{align*}
\diver\mathbf{v}&=-\left(\frac{n-p}{p}\right)^{p-1}\Delta_p v \\
&=u^{-\frac{n(p-1)}{n-p}}\Delta_p u -\frac{n(p-1)}{n-p}u^{-\frac{p(n-1)}{n-p}}|\nabla u|^p\\
&=-u^{\frac{p}{n-p}} -\frac{n(p-1)}{n-p}u^{-\frac{p(n-1)}{n-p}}|\nabla u|^p \in  C^{0,\alpha}_{\text{loc}}(\RR^n).
\end{align*}
By standard elliptic regularity, we have $v\in C^{2,\alpha}_{\text{loc}}(\Omega_0)$, $u\in C^{2,\alpha}(\Omega_0)$ and then $\diver\mathbf{v}\in C^{1,\alpha}_{\text{loc}}(\Omega_0)$. Differentiating \eqref{eq12}, we get
$$
\partial_i \left(\diver\mathbf v\right) =n\,\partial_i\left(\diver \mathbf{v}\right).
$$
Therefore $\diver\mathbf v=\text{const}$ on $\Omega_0$ and thus $\mathbf{v}=C(x-x_0)$, for some $C\in\RR$ and some $x_0\in\RR^n$. Thus
$$
v=C_1+C_2|x-x_0|^{\frac{p}{p-1}}
$$
on $\Omega_0$, for some $C_1,C_2>0$. Then $u(x)=\mathcal U_{\lambda,x_0} (x)$ on $\Omega_0$ for some $\lambda>0$ and $x_0\in\RR^n$. Since the argument above holds whenever $\nabla u\neq 0$, we must have $\Omega_0=\RR^n\setminus\{x_0\}$ and the result follows.

\

\subsection{Proof of Theorem \ref{t-e1} (ii)} Let $u$ be a positive weak solution of equation \eqref{eq-pl} with $\frac{2n}{n+1}<p<2$. From Corollary \ref{c-fond} with $l=2$ we have
\begin{equation}\label{est3}
 \int u^{\frac{(n-1)p}{n-p}}|\mathring{\mathbf{V}}|^2\eta^2\leq C\int u^{\frac{(2-p)n-p}{n-p}}|\nabla u|^{2(p-1)} |\nabla  \eta|^2.
\end{equation}
We choose the same cutoff functions as in (i). Since $\frac{2n}{n+1}<p<2$, we have
$$
\frac{(2-p)n-p}{n-p}< 0\qquad\text{and}\qquad 2(p-1)<p.
$$
Thanks to Lemma \ref{l-ser} we have $u\geq C R^{-\frac{n-p}{p-1}}$ on $A_R$ and using H\"older inequality we get
\begin{align*}
\int u^{\frac{(2-p)n-p}{n-p}}|\nabla u|^{2(p-1)} |\nabla  \eta|^2 &\leq C R^{-\frac{(n-1)(2-p)}{p-1}} \left(\int_{A_R}|\nabla u|^p\right)^{\frac{2(p-1)}{p}}|A_R|^{\frac{2-p}{p}}\\
&\leq CR^{-\frac{(n-p)(2-p)}{p(p-1)}}  E^{\text{kin}}_{A_R}(u)^{\frac{2(p-1)}{p}}\\
&\leq CR^{-\frac{(n-p)(2-p)}{p(p-1)}}  E_{A_R}(u)^{\frac{2(p-1)}{p}}
\end{align*}
Thanks to the energy assumption, we have  that the righthand side of \eqref{est3} is uniformly bounded in $R$. Hence
$$
\int_{\RR^n} u^{\frac{(n-1)p}{n-p}}|\mathring{\mathbf{V}}|^2=0,
$$
and the conclusion follows as in the proof of Theorem \ref{t-e1} (i).

\subsection{Proof of Theorem \ref{t-e1} (iii)} Let $u$ be a positive weak solution of equation \eqref{eq-pl} with $2<p<n$ and assume
$$u(x)\leq C|x|^{\alpha},$$
as $|x|\to\infty$ for some $\alpha\geq 0$.  From Corollary \ref{c-fond} with $l=2$ we have
\begin{equation}\label{est2}
 \int u^{\frac{(n-1)p}{n-p}}|\mathring{\mathbf{V}}|^2\eta^2\leq C\int u^{\frac{(2-p)n-p}{n-p}}|\nabla u|^{2(p-1)} |\nabla  \eta|^2.
\end{equation}
We choose the same cutoff functions as in (i). We have
\begin{align*}
\int u^{\frac{(2-p)n-p}{n-p}}|\nabla u|^{2(p-1)} |\nabla  \eta|^2 &=\int u^{\frac{(2-p)n-p}{n-p}}|\nabla u|^{\theta}|\nabla u|^{2(p-1)-\theta}  |\nabla  \eta|^2 \\
&\leq C R^{\left(\frac{1}{n-p}+\eps\right)\alpha\theta}\int u^{\frac{(2-p)n-p+\theta(n-1)-\eps(n-p)\theta}{n-p}}|\nabla u|^{2(p-1)-\theta}  |\nabla  \eta|^2,
\end{align*}
for every $\theta>0$ and every $\eps>0$ small enough, where we used the gradient estimate in Corollary \ref{c-ge}. We assume
$$
p-2<\theta<2(p-1)\quad\text{and}\quad \theta>\frac{(p-2)n+p}{n-1-\eps(n-p)}
$$
and we apply H\"older inequality to obtain
\begin{align*}
\int u^{\frac{(2-p)n-p}{n-p}}&|\nabla u|^{2(p-1)} |\nabla  \eta|^2\\
&\leq C R^{\left(\frac{1}{n-p}+\eps\right)\alpha\theta-2}\left(\int_{A_R}u^{p^*}\right)^{\frac{\theta(n-1)-(p-2)n-p-\eps(n-p)\theta}{np}}\left(\int_{A_R}|\nabla u|^p\right)^{\frac{2(p-1)-\theta}{p}}|A_R|^{\frac{p+\theta+\eps(n-p)\theta}{np}}\\
&\leq C R^{\left(\frac{1}{n-p}+\eps\right)\alpha\theta-1+\frac{\theta+\eps(n-p)\theta}{p}} E^{\text{pot}}_{A_R}(u)^{\frac{\theta(n-1)-(p-2)n-p-\eps(n-p)\theta}{np}}E^{\text{kin}}_{A_R}(u)^{\frac{2(p-1)-\theta}{p}}\\
&\leq C R^{\left(\frac{1}{n-p}+\eps\right)\alpha\theta-1+\frac{\theta+\eps(n-p)\theta}{p}}  E_{A_R}(u)^{\frac{np-p-\theta-\eps(n-p)\theta}{np}}
\end{align*}
Under our assumptions, $E_{A_R}(u)\leq C R^k$ for some $k>0$. Then
$$
\int u^{\frac{(2-p)n-p}{n-p}}|\nabla u|^{2(p-1)} |\nabla  \eta|^2 \leq C R^{\left(\frac{1}{n-p}+\eps\right)\alpha\theta-1+\frac{\theta+\eps(n-p)\theta}{p}+k\left(\frac{np-p-\theta-\eps(n-p)\theta}{np}\right)}.
$$
Since $p-2<\frac{(p-2)n+p}{n-1-\eps(n-p)}<2(p-1)$, by choosing $\theta$ close to $\frac{(p-2)n+p}{n-1-\eps(n-p)}$ and $\eps$ close to $0$, thanks to the energy assumption, we have  that the righthand side of \eqref{est2} is uniformly bounded in $R$. Hence
$$
\int_{\RR^n} u^{\frac{(n-1)p}{n-p}}|\mathring{\mathbf{V}}|^2=0,
$$
and the conclusion follows as in the proof of Theorem \ref{t-e1} (i).

\

\section{Rigidity with control at infinity}\label{s-ci}

In this section we prove Theorems \ref{t-se1} and \ref{t-se2}. We start by proving the following weak energy estimate which shows that a weighted energy has controlled growth on balls.
\begin{lemma}\label{l-t} Let $u$ be a positive weak solution of equation \eqref{eq-pl} for some $1<p<n$ and let $t<-1$. Then, for every $R>1$, we have
$$
\int_{B_R}u^{\frac{np+t(n-p)}{n-p}}+\int_{B_R}u^t|\nabla u|^p \leq C R^\beta,
$$
for some $C=C(n,p,t)>0$ and
$$
\beta:=\begin{cases}
-\frac{t(n-p)}{p}& t+p> 0\\
-\frac{(t+1)(n-p)}{p-1}& t+p\leq 0.
\end{cases}
$$
\end{lemma}
\begin{proof}  We choose $\eta\in C^{\infty}_0(\RR^n)$ be such that $\eta\equiv 1$ in $B_R$, $\eta\equiv 0$ in $B_{2R}^c$, $0\leq \eta\leq 1$ on $\RR^n$ and $\eta$ satisfies
$$
|\nabla \eta|^{2} \leq C R^{-2} \quad\text{in }A_R=B_{2R}\setminus B_{R}.
$$
Testing the weak formulation given in \ref{deb} with $u^{t+1} \eta^l$, for $l$ sufficiently large, we obtain
\begin{align*}
-\int u^{\frac{np+t(n-p)}{n-p}}\eta^l &= -(t+1)\int u^t |\nabla u|^p \eta^l -l\int u^{t+1}|\nabla u|^{p-2}\langle \nabla u, \nabla \eta\rangle\eta^{l-1}\\
&\geq |t+1| \int u^t |\nabla u|^p \eta^l -l \int u^{t+1}|\nabla u|^{p-1}|\nabla \eta|\eta^{l-1}
\end{align*}
If $t+p> 0$, we get
\begin{align*}
-\int u^{\frac{np+t(n-p)}{n-p}}\eta^l &\geq |t+1| \int u^t |\nabla u|^p \eta^l -\eps \int u^{\frac{np+t(n-p)}{n-p}}\eta^l-\eps \int u^t |\nabla u|^p \eta^l\\
&\quad-C_{\eps} \int |\nabla \eta|^{\frac{np+t(n-p)}{p}}\eta^{l-\frac{np+t(n-p)}{p}}\\
&\geq \left(|t+1|-\eps\right) \int u^t |\nabla u|^p \eta^l -\eps \int u^{\frac{np+t(n-p)}{n-p}}\eta^l - C_\eps R^{-\frac{np+t(n-p)}{p}}|B_{2R}|
\end{align*}
for some $\eps>0$, where we used Cauchy-Schwarz and Young's inequalities\footnote{We remark that on the third term $u^{t+1}|\nabla u|^{p-1}|\nabla \eta|\eta^{l-1}$ we used the following generalization of the classical Young's inequality:
$$
abc\leq\eps a^r+ \eps b^s + k(\varepsilon)c^t\, ,
$$
for all $a,b,c\geq 0$, $\varepsilon>0$ and where $r,s,t>1$ are such that $\frac{1}{r}+\frac{1}{s}+\frac{1}{t}=1$.}. Choosing $\eps$ small enough, we conclude.

\

On the other hand, if $t+p\leq0$, we get
\begin{align*}
-\int u^{\frac{np+t(n-p)}{n-p}}\eta^l &\geq |t+1| \int u^t |\nabla u|^p \eta^l -\eps \int u^t |\nabla u|^p \eta^l-C_{\eps} \int u^{t+p}|\nabla \eta|^{p}\eta^{l-p}\\
&\geq \left(|t+1|-\eps\right) \int u^t |\nabla u|^p \eta^l - C_\eps R^{-\frac{(t+p)(n-p)}{p-1}-p+n}\\
&\geq \left(|t+1|-\eps\right) \int u^t |\nabla u|^p \eta^l - C_\eps R^{-\frac{(t+1)(n-p)}{p-1}}
\end{align*}
for some $\eps>0$, where we used Lemma \ref{l-ser}. Choosing $\eps$ small enough, we conclude.

\end{proof}

\subsection{Proof of Theorems \ref{t-se0} and \ref{t-se1}} Let $u$ be a positive weak solution of equation \eqref{eq-pl} with $1<p<2$ and let $\gamma\geq 0$. Let $\eta$ as in the proof of the previous lemma. From Corollary \ref{c-fond} with $l=2$ we have
\begin{equation}\label{est2b}
 \int u^{\frac{(n-1)p}{n-p}}|\mathring{\mathbf{V}}|^2\eta^2\leq C\int u^{\frac{(2-p)n-p}{n-p}}|\nabla u|^{2(p-1)} |\nabla  \eta|^2.
\end{equation}
We choose $\eta\in C^{\infty}_0(\RR^n)$ be such that $\eta\equiv 1$ in $B_R$, $\eta\equiv 0$ in $B_{2R}^c$, $0\leq \eta\leq 1$ on $\RR^n$ and $\eta$ satisfies
$$
|\nabla \eta|^{2} \leq C R^{-2} \quad\text{in }A_R=B_{2R}\setminus B_{R}.
$$
Since $p<2$, then $2(p-1)<p$ and, by H\"older inequality we obtain
\begin{align*}
\int u^{\frac{(2-p)n-p}{n-p}}&|\nabla u|^{2(p-1)} |\nabla  \eta|^2 \\
&\leq \sup_{A_R} u^{\gamma} \int_{A_R}u^{\frac{(2-p)n-p}{n-p}-\gamma}|\nabla u|^{2(p-1)} |\nabla  \eta|^2 \\
&\leq \frac{1}{R^2}\sup_{A_R} u^{\gamma} \int_{A_R}u^{\frac{(2-p)n-p}{n-p}-\gamma-\frac{2t(p-1)}{p}}\left(u^t|\nabla u|^p\right)^{\frac{2(p-1)}{p}} \\
&\leq  \frac{1}{R^2}\sup_{A_R} u^{\gamma} \left(\int_{A_R}u^t|\nabla u|^p\right)^{\frac{2(p-1)}{p}}\left(\int_{A_R}u^{\frac{(2-p)np-p^2-\gamma p (n-p)-2t(p-1)(n-p)}{(n-p)(2-p)}}\right)^{\frac{2-p}{p}}.
\end{align*}
We choose
$$
t=\bar{t}:=-\frac{p}{n-p}-\gamma
$$
in order to have
$$
\frac{(2-p)np-p^2-\gamma p (n-p)-2\bar{t}(p-1)(n-p)}{(n-p)(2-p)} = \frac{np+\bar{t}(n-p)}{n-p}.
$$
We observe that
\begin{equation}\label{t1}
\bar{t}<-1\qquad\Longleftrightarrow\qquad (\gamma-2)p<(\gamma-1) n.
\end{equation}
Then, from Lemma \ref{l-t}, we obtain
\begin{align}\label{est12}\nonumber
\int u^{\frac{(2-p)n-p}{n-p}}|\nabla u|^{2(p-1)} |\nabla  \eta|^2&\leq  \frac{C}{R^2}\sup_{A_R} u^{\gamma} \left(\int_{A_R}u^{\bar{t}}|\nabla u|^p\right)^{\frac{2(p-1)}{p}}\left(\int_{A_R}u^{\frac{np+\bar{t}(n-p)}{n-p}}\right)^{\frac{2-p}{p}}\\
&\leq C\sup_{A_R} u^{\gamma} R^{\beta-2}.
\end{align}

\subsection{Proof of Theorem \ref{t-se0} (i)} Let $n=2$, $1<p<2$ and choose $\gamma=0$. In particular \eqref{t1} is satisfied and $\bar{t}+p=-\frac{p(p-1)}{2-p}<0$. Then, $\beta=2$ and from \eqref{est12} we get
$$
\int u^{\frac{4-3p}{2-p}}|\nabla u|^{2(p-1)} |\nabla  \eta|^2\leq C,
$$
for every $R>0$. Hence, arguing as in the proof of Theorem \ref{t-e1} (i), from Corollary \ref{c-fond} the conclusion follows.

\

\subsection{Proof of Theorem \ref{t-se0} (ii)} Let $n=3$. If $\frac 32<p<2$ we again choose $\gamma=0$. In particular \eqref{t1} is satisfied  and $\bar{t}+p=\frac{p(2-p)}{3-p}>0$. Then, $\beta=1$ and from \eqref{est12} we get
$$
\int u^{\frac{6-4p}{3-p}}|\nabla u|^{2(p-1)} |\nabla  \eta|^2\leq CR^{-1}\longrightarrow 0,
$$
as $R$ tends to $\infty$. Hence, arguing as in the proof of Theorem \ref{t-e1} (i), from Corollary \ref{c-fond} the conclusion follows.

\

\subsection{Proof of Theorem \ref{t-se1} (i)} If $1<p\leq \frac32$ we assume
$$u(x)\leq C|x|^{\alpha},$$
as $|x|\to\infty$, for some $\alpha<\bar{\alpha}:=\frac{3(p-1)(3-p)}{p(3-2p)}$. If $\gamma$ satisfies \eqref{t1}, from \eqref{est12} we get
$$
\int u^{\frac{6-4p}{3-p}}|\nabla u|^{2(p-1)} |\nabla  \eta|^2\leq CR^{\beta-2+\alpha\gamma}.
$$
We choose $\alpha$ and $\gamma$ such that
$$
\begin{cases}
(\gamma-2)p<3(\gamma-1) \\
\beta-2+\alpha\gamma\leq 0
\end{cases} \quad\Longleftrightarrow \quad\begin{cases}
\gamma>\frac{3-2p}{3-p}\\
\alpha\leq \frac{2-\beta}{\gamma}.
\end{cases}
$$
For $\gamma$ close to $\frac{3-2p}{3-p}$, we have $\bar{t}+p$ close to $p-1>0$. Then
$$
\beta=1+\frac{\gamma(3-p)}{p}\quad\text{and}\quad \alpha\leq\frac{p(\gamma+1)-3\gamma}{p\gamma}.
$$
Since the right-hand side in the second inequality is decreasing in $\gamma$, it is sufficient to have
$$
\alpha<\frac{3(p-1)(3-p)}{p(3-2p)}=\bar{\alpha}.
$$
Letting $R\to\infty$, from \eqref{est2b} we obtain
$$
\int_{\RR^n} u^{\frac{2p}{3-p}}|\mathring{\mathbf{V}}|^2=0,
$$
and the conclusion follows as in the proof of Theorem \ref{t-e1} (i).

\

\subsection{Proof of Theorem \ref{t-se1} (ii)} Let $n\geq 4$, $1<p<2$ and assume
$$u(x)\leq C|x|^{\alpha},$$
as $|x|\to\infty$, for some $\alpha<\bar{\alpha}:=\frac{(3p-n)(n-p)}{p(n-2p)}$. If $\gamma$ satisfies \eqref{t1}, from \eqref{est12} we get
$$
\int u^{\frac{(2-p)n-p}{n-p}}|\nabla u|^{2(p-1)} |\nabla  \eta|^2\leq CR^{\beta-2+\alpha\gamma}.
$$
We choose $\alpha$ and $\gamma$ such that
$$
\begin{cases}
(\gamma-2)p<n(\gamma-1) \\
\beta-2+\alpha\gamma\leq0
\end{cases} \quad\Longleftrightarrow\quad\begin{cases}
\gamma>\frac{n-2p}{n-p}\\
\alpha\leq\frac{2-\beta}{\gamma}.
\end{cases}
$$
For $\gamma$ close to $\frac{n-2p}{n-p}$, we have $\bar{t}+p$ close to $p-1>0$. Then
$$
\beta=1+\frac{\gamma(n-p)}{p}\quad\text{and}\quad \alpha\leq\frac{p(\gamma+1)-n\gamma}{p\gamma}.
$$
Since the right-hand side in the second inequality is decreasing in $\gamma$, it is sufficient to have
$$
\alpha<\frac{(3p-n)(n-p)}{p(n-2p)}=\bar{\alpha}.
$$
Letting $R\to\infty$, from \eqref{est2b} we obtain
$$
\int_{\RR^n} u^{\frac{(n-1)p}{n-p}}|\mathring{\mathbf{V}}|^2=0,
$$
and the conclusion follows as in the proof of Theorem \ref{t-e1} (i).

\

\subsection{Proof of Theorem \ref{t-se2}}
Let $p>2$ and assume
$$u(x)\leq C|x|^{\alpha},$$
as $|x|\to\infty$. We choose $\eta\in C^{\infty}_0(\RR^n)$ be such that $\eta\equiv 1$ in $B_R$, $\eta\equiv 0$ in $B_{2R}^c$, $0\leq \eta\leq 1$ on $\RR^n$ and $\eta$ satisfies
$$
|\nabla \eta|^{2} \leq C R^{-2} \quad\text{in }A_R=B_{2R}\setminus B_{R}.
$$

\

\noindent {\em Case 1:} $\alpha\geq 0$. We have
\begin{align*}
\int u^{\frac{(2-p)n-p}{n-p}}&|\nabla u|^{2(p-1)} |\nabla  \eta|^2 =\int u^{\frac{(2-p)n-p}{n-p}}|\nabla u|^{\theta}|\nabla u|^{2(p-1)-\theta}  |\nabla  \eta|^2 \\
&\leq \frac{C}{R^2}\sup_{A_R}u^{\gamma}\int u^{\frac{(2-p)n-p-\gamma(n-p)}{n-p}}|\nabla u|^{\theta}||\nabla u|^{2(p-1)-\theta}  \\	
&\leq CR^{\left(\frac{1}{n-p}+\eps\right)\alpha\theta-2}\sup_{A_R}u^{\gamma}\int u^{\frac{(2-p)n-p-\gamma(n-p)+\theta(n-1)-\eps(n-p)\theta}{n-p}}|\nabla u|^{2(p-1)-\theta}\\
&= CR^{\left(\frac{1}{n-p}+\eps\right)\alpha\theta-2}\sup_{A_R}u^{\gamma}\int u^{\frac{(2-p)n-p-\gamma(n-p)+\theta(n-1)-\eps(n-p)\theta}{n-p}-\frac{t[2(p-1)-\theta]}{p}}\left(u^t|\nabla u|^p\right)^{\frac{2(p-1)-\theta}{p}},
\end{align*}
for every $t\in\RR$, $\theta>0$, $\gamma\geq 0$ and every $\eps>0$ small enough, where we used the gradient estimate in Corollary \ref{c-ge} with $\alpha\geq 0$. We assume
\begin{equation}\label{t2b}
p-2<\theta<2(p-1)
\end{equation}
and we apply H\"older inequality to obtain
\begin{align*}
\int &u^{\frac{(2-p)n-p}{n-p}}|\nabla u|^{2(p-1)} |\nabla  \eta|^2 \\
&\leq CR^{\left(\frac{1}{n-p}+\eps\right)\alpha\theta-2}\sup_{A_R}u^{\gamma}\left(\int_{A_R} u^{\left\{\frac{(2-p)n-p-\gamma(n-p)+\theta(n-1)-\eps(n-p)\theta}{n-p}-\frac{t[2(p-1)-\theta]}{p}\right\}\frac{p}{2-p+\theta}}\right)^{\frac{2-p+\theta}{p}}\cdot\\
&\hspace{6cm}\cdot\left(\int_{A_R} u^t|\nabla u|^p\right)^{\frac{2(p-1)-\theta}{p}}.
\end{align*}
We choose
$$
t=\bar{t}:=-\frac{p+\theta+\eps(n-p)\theta}{n-p}-\gamma
$$
in order to have
$$
\left\{\frac{(2-p)n-p-\gamma(n-p)+\theta(n-1)-\eps(n-p)\theta}{n-p}-\frac{\bar{t}[2(p-1)-\theta]}{p}\right\}\frac{p}{2-p+\theta} = \frac{np+\bar{t}(n-p)}{n-p}.
$$
We observe that
\begin{equation}\label{t1b}
\bar{t}<-1\qquad\Longleftrightarrow\qquad (\gamma-2)p<(\gamma-1) n+\theta+\eps(n-p)\theta.
\end{equation}
Then, from Lemma \ref{l-t}, we obtain
\begin{align}\label{est12b}\nonumber
\int u^{\frac{(2-p)n-p}{n-p}}&|\nabla u|^{2(p-1)} |\nabla  \eta|^2\\\nonumber
&\leq CR^{\left(\frac{1}{n-p}+\eps\right)\alpha\theta-2}\sup_{A_R} u^{\gamma} \left(\int_{A_R}u^{\bar{t}}|\nabla u|^p\right)^{\frac{2(p-1)-\theta}{p}}\left(\int_{A_R}u^{\frac{np+\bar{t}(n-p)}{n-p}}\right)^{\frac{2-p+\theta}{p}}\\\nonumber
&\leq C\sup_{A_R} u^{\gamma} R^{\beta+\left(\frac{1}{n-p}+\eps\right)\alpha\theta-2}\\
&\leq C R^{\beta+\left(\frac{1}{n-p}+\eps\right)\alpha\theta-2+\alpha\gamma}
\end{align}
We aim at finding $\theta$  and $\gamma$ satisfying \eqref{t2b} and \eqref{t1b} such that
\begin{equation}\label{ttt}
\beta+\left(\frac{1}{n-p}+\eps\right)\alpha\theta-2+\alpha\gamma\leq0.
\end{equation}
{\em Case 1.1:} If $p\geq\frac{n+2}{3}$, we choose $\theta$ close to $p-2$, $\gamma=0$ and $\eps>0$ small enough. Then $\bar{t}+p$ is close to
$$\frac{-p^2+(n-2)p+2}{n-p}$$
If $\frac{n+2}{3}\leq p<\frac{n-2+\sqrt{n^2-4n+12}}{2}:=\check{p}$, then $\bar{t}+p>0$ and hence $\beta=-\frac{\bar{t}(n-p)}{p}.$
A simple computation shows that \eqref{ttt} is satisfied if
$$
\alpha< \frac{2(n-p)}{p(p-2)}=:\hat{\alpha}.
$$
On the other hand, if $p\geq \check{p}$, then $\bar{t}+p<0$ and hence $\beta=-\frac{(\bar{t}+1)(n-p)}{p-1},$ which is close to $\frac{3p-n-2}{p-1}$. Thus \eqref{ttt} is satisfied if
$$
\alpha< \frac{(n-p)^2}{(p-2)(p-1)}=:\check{\alpha}.
$$

\

\noindent{\em Case 1.2:} If $2<p<\frac{n+2}{3}$, we choose again $\theta$ close to $p-2$, $\gamma$ close to $\frac{n-3p+2}{n-p}$ and $\eps>0$ small enough. Then $\bar{t}+p$ is close to $p-1>0$ and hence $\beta=-\frac{\bar{t}(n-p)}{p}$ and it is close to $\frac{n-p}{p}$. Hence, in order to verify \eqref{ttt} it is sufficient to choose
$$
\alpha<\frac{(3p-n)(n-p)}{p(n-2p)}=:\bar{\alpha}.
$$
In particular this case occurs only if $\frac n 3<p<\frac{n+2}{3}$.

\

\noindent {\em Case 2:} $\alpha< 0$. We have
\begin{align*}
\int u^{\frac{(2-p)n-p}{n-p}}&|\nabla u|^{2(p-1)} |\nabla  \eta|^2 =\int u^{\frac{(2-p)n-p}{n-p}}|\nabla u|^{\theta}|\nabla u|^{2(p-1)-\theta}  |\nabla  \eta|^2 \\
&\leq \frac{C}{R^2}\sup_{A_R}u^{\gamma}\int u^{\frac{(2-p)n-p-\gamma(n-p)}{n-p}}|\nabla u|^{\theta}|\nabla u|^{2(p-1)-\theta}  \\	
&\leq CR^{-2}\sup_{A_R}u^{\gamma}\int u^{\frac{(2-p)n-p-\gamma(n-p)+\theta(n-1)-\eps(n-p)\theta}{n-p}}|\nabla u|^{2(p-1)-\theta}\\
&= CR^{-2}\sup_{A_R}u^{\gamma}\int u^{\frac{(2-p)n-p-\gamma(n-p)+\theta(n-1)-\eps(n-p)\theta}{n-p}-\frac{t[2(p-1)-\theta]}{p}}\left(u^t|\nabla u|^p\right)^{\frac{2(p-1)-\theta}{p}},
\end{align*}
for every $t\in\RR$, $\theta>0$, $\gamma\geq 0$ and every $\eps>0$ small enough, where we used the gradient estimate in Corollary \ref{c-ge} with $\alpha< 0$. We assume
\begin{equation}\label{t2c}
p-2<\theta<2(p-1)
\end{equation}
and we apply H\"older inequality to obtain
\begin{align*}
\int &u^{\frac{(2-p)n-p}{n-p}}|\nabla u|^{2(p-1)} |\nabla  \eta|^2 \\
&\leq CR^{-2}\sup_{A_R}u^{\gamma}\left(\int_{A_R} u^{\left\{\frac{(2-p)n-p-\gamma(n-p)+\theta(n-1)-\eps(n-p)\theta}{n-p}-\frac{t[2(p-1)-\theta]}{p}\right\}\frac{p}{2-p+\theta}}\right)^{\frac{2-p+\theta}{p}}\cdot\\
&\hspace{6cm}\cdot\left(\int_{A_R} u^t|\nabla u|^p\right)^{\frac{2(p-1)-\theta}{p}}.
\end{align*}
We choose
$$
t=\bar{t}:=-\frac{p+\theta+\eps(n-p)\theta}{n-p}-\gamma
$$
in order to have
$$
\left\{\frac{(2-p)n-p-\gamma(n-p)+\theta(n-1)-\eps(n-p)\theta}{n-p}-\frac{\bar{t}[2(p-1)-\theta]}{p}\right\}\frac{p}{2-p+\theta} = \frac{np+\bar{t}(n-p)}{n-p}.
$$
We observe that
\begin{equation}\label{t1c}
\bar{t}<-1\qquad\Longleftrightarrow\qquad (\gamma-2)p<(\gamma-1) n+\theta+\eps(n-p)\theta.
\end{equation}
Then, from Lemma \ref{l-t}, we obtain
\begin{align}\label{est12c}\nonumber
\int u^{\frac{(2-p)n-p}{n-p}}&|\nabla u|^{2(p-1)} |\nabla  \eta|^2\\\nonumber
&\leq CR^{-2}\sup_{A_R} u^{\gamma} \left(\int_{A_R}u^{\bar{t}}|\nabla u|^p\right)^{\frac{2(p-1)-\theta}{p}}\left(\int_{A_R}u^{\frac{np+\bar{t}(n-p)}{n-p}}\right)^{\frac{2-p+\theta}{p}}\\\nonumber
&\leq C\sup_{A_R} u^{\gamma} R^{\beta-2}\\
&\leq CR^{\beta-2+\alpha\gamma}
\end{align}
We aim at finding $\theta$  and $\gamma$ satisfying \eqref{t2c} and \eqref{t1c} such that
\begin{equation}\label{ttt2}
\beta-2+\alpha\gamma<0.
\end{equation}
{\em Case 2.1:} If $p\geq\frac{n+2}{3}$ the choice $\theta$ close to $p-2$, $\gamma=0$ and $\eps>0$ small enough gives that $\beta<2$ both if $\bar{t}+p>0$ and if $\bar{t}+p\leq 0$. Since $\alpha<0$, \eqref{ttt2} is satisfied.

\

\noindent {\em Case 2.2:} If $2<p<\frac{n+2}{3}$, we choose again $\theta$ close to $p-2$, $\gamma$ close to $\frac{n-3p+2}{n-p}$ and $\eps>0$ small enough. Then $\bar{t}+p$ is close to $p-1>0$ and hence $\beta=-\frac{\bar{t}(n-p)}{p}$ and it is close to $\frac{n-p}{p}$. Hence a simple computation shows that, in order to verify \eqref{ttt2}, it is sufficient to choose
$$
\alpha<\frac{(3p-n)(n-p)}{p(n-3p+2)}:=\tilde{\alpha}.
$$
To conclude we observe that we have
$$
\int u^{\frac{(2-p)n-p}{n-p}}|\nabla u|^{2(p-1)} |\nabla  \eta|^2\leq C R^{-\delta}
$$
for some $\delta>0$ if one of the assumptions
\begin{itemize}
\item[(i)] $n=3$, and $2<p<3$ and  $\alpha<\check{\alpha}$;
\item[(ii)] $n=4$, and  $$2<p<\check{p}\quad\text{and}\quad\alpha<\hat{\alpha},$$ or $$\check{p}\leq p <4\quad\text{and}\quad\alpha<\check{\alpha};$$
\item[(iii)] $n=5$ or $n=6$, and $$2<p<\tfrac{n+2}{3}\quad\text{and}\quad\alpha<\bar{\alpha},$$ or $$\tfrac{n+2}{3}\leq p <\check{p}\quad\text{and}\quad\alpha<\hat{\alpha},$$ or
$$\check{p}\leq p <n\quad\text{and}\quad\alpha<\check{\alpha};$$
\item[(iv)] $n\geq 7$ and $$2<p\leq\tfrac{n}{3}\quad\text{and}\quad\alpha<\tilde{\alpha},$$ or $$\tfrac{n}{3}<p<\tfrac{n+2}{3}\quad\text{and}\quad\alpha<\bar{\alpha},$$ or $$\tfrac{n+2}{3}\leq p <\check{p}\quad\text{and}\quad\alpha<\hat{\alpha},$$ or
$$\check{p}\leq p <n\quad\text{and}\quad\alpha<\check{\alpha},$$
\end{itemize}

hold. Then letting $R\to\infty$, from \eqref{est2b} we obtain
$$
\int_{\RR^n} u^{\frac{(n-1)p}{n-p}}|\mathring{\mathbf{V}}|^2=0,
$$
and the conclusion follows as in the proof of Theorem \ref{t-e1} (i).

\

\appendix

\section{Riemannian setting}\label{Riem}

In this Appendix we consider the case of a complete, non-compact (without boundary) Riemannian manifold $(M^n,g)$ of dimension $n\geq 2$. We will emphasize the main differences with respect to Euclidean case when dealing with the same issues.

We consider positive weak solutions of
\begin{equation}\label{eq-pl-Riem}
\Delta_{p} u + u^{p^*-1} = 0 \quad \text{ in } M^n
\end{equation}
where $\Delta_{p}$ is the usual $p-$Laplace-Beltrami operator with respect to the metric $g$. Moreover, we denote with $\mathrm{Ric}$ and $\mathrm{Sec}$ the Ricci and the sectional curvatures of $(M^n,g)$, respectively.

When $\mathrm{Ric}\geq 0$ equation \eqref{eq-pl-Riem} has been recently studied, in the semilinear case $p=2$, in \cite{catmon} and \cite{FMM}; in particular, in \cite{catmon} the authors prove that the only positive classical solutions to \eqref{eq-pl-Riem} are given by the Aubin-Talenti bubbles \eqref{tal} with $p=2$ and the Riemannian manifold is isometric to the Euclidean space, provided $n=3$ or $u$ has finite energy or $u$ satisfies suitable conditions at infinity.  Furthermore, when the manifold is a Cartan-Hadamard manifold, i.e.  complete and simply connected Riemannian manifold with non positive sectional curvature, in \cite{muso} the authors prove that all the positive energy minimizing solutions to \eqref{eq-pl-Riem} are given by the Aubin-Talenti bubbles \eqref{tal} and the Riemannian manifold is isometric to the Euclidean space, assuming the validity of an optimal isoperimetric inequality on $M^n$.

In this appendix we deal with the quasilinear case, i.e. $1<p<n$, and we show that the analogue of Theorems \ref{t-e1}-\ref{t-se0}-\ref{t-se1}-\ref{t-se2} hold, provided that the Riemannian manifold $(M^n,g)$ satisfies:
\begin{itemize}
\item[(i)] $\mathrm{Ric}\geq 0$, if $1<p<2$,  or
\item[(ii)] $\mathrm{Sec}\geq 0$, if $2<p<n$.
\end{itemize}
Of course, if $n=2$, both conditions are replaced by non-negativity of the scalar/Gauss curvature. The main differences with respect to the Euclidean case are the following:
\begin{itemize}
	\item By assumption, we choose $u$ to be a positive weak solution having the regularity given in \eqref{eq-r1}--\eqref{eq-r4}.
	\item The estimate in Corollary \ref{c-fond} becomes the following
\begin{equation}\label{eq-ap2}
  \int u^{\frac{(n-1)p}{n-p}}|\mathring{\mathbf{V}}|^2\eta^l+ \int u^{\frac{(n-1)p}{n-p}}\mathrm{Ric}(\mathbf{v},\mathbf{v})\eta^l \leq C\int u^{\frac{(2-p)n-p}{n-p}}|\nabla u|^{2(p-1)} |\nabla  \eta|^2\eta^{l-2}\, ,
\end{equation}
the Ricci tensor appears when computing
\begin{align*}
\diver(\mathbf{v}\cdot \nabla \mathbf{v})&= \nabla_j (\nabla_i \mathbf{v}_j \mathbf{v}_i)=\nabla_j\nabla_i \mathbf{v}_j \mathbf{v}_i + \nabla_i \mathbf{v}_j \nabla_j \mathbf{v}_i \\
&=\nabla_i\nabla_j \mathbf{v}_j \mathbf{v}_i - \mathrm{Ric}(\mathbf{v},\mathbf{v}) +|\nabla \mathbf{v}|^2
\end{align*}
since $\nabla_j \mathbf{v}_i$ is symmetric. This identity is used in the proof of Propositions 6.2 and 7.1 in \cite{serzou}, which are used in our proof of Corollary \ref{c-fond}.

\item Let $r(x):=\mathrm{dist}(x,p)$ be the geodesic distance from a fixed point $p\in M$. Due to the presence of the curvature, when $\mathrm{Ric}\geq 0$ the Laplacian comparison implies that the function $r^{-\frac{n-p}{p-1}}$ is a weak $p-$subharmonic function, i.e. its $p-$Laplacian is non-negative. Hence, Lemma \ref{l-ser} holds true also in this setting.

\item Lemma \ref{l-dario} and Lemma \ref{l-2} still hold, provided
$$
\mathrm{Vol}(B_R)\leq C R^n
$$
for every $R>0$, which is ensured by Bishop-Gromov volume comparison.

\item The Bocher formula in Lemma \ref{pboch} becomes the following (see e.g. \cite{val})
\begin{align*}
\frac1p P_f(|\nabla f|^p) &\geq \frac1n (\Delta_p f)^2 +\frac{n}{n-1}\left(\frac1n \Delta_p f - (p-1)|\nabla f|^{p-4}\nabla^2 f (\nabla f, \nabla f)\right)^2\\
&\quad+ |\nabla f|^{p-2}\left[\langle \nabla f, \nabla \Delta_p f\rangle - (p-2) \frac{\Delta_p f}{|\nabla f|^2}\nabla^2 f (\nabla f, \nabla f)  \right] \\
&\quad + |\nabla f|^{2(p-2)}\mathrm{Ric}(\nabla f,\nabla f) \, .
\end{align*}
Thus, Lemma \ref{pboch} still holds thanks to the condition $\mathrm{Ric}\geq 0$.

\item Following the proof of the gradient estimate in Proposition \ref{p-ge} (in particular, the estimate above \eqref{eq11}), one can observe that the same arguments works if we can construct a family of smooth cut-off functions $\phi$ defined on $B_{2R}$ such that
\begin{equation}\label{eq-ap}
|\nabla \phi|\leq \frac{C}{R}\phi^{1-\delta},\qquad |\nabla f|^2 \Delta\phi+(p-2)\nabla^2\phi(\nabla f,\nabla f)\geq -\frac{C}{R^2}\phi^{1-2\delta}|\nabla f|^2.
\end{equation}
Assume that $(M^n,g)$ has $\mathrm{Sec}\geq 0$. Let $r(x):=\mathrm{dist}(x,p)$ be the geodesic distance from a fixed point $p\in M$ and let $\psi\in C^2([0,\infty))$ be such that $\psi\equiv 1$ in $[0,1)$, $\psi\equiv 0$ in $[2,\infty)$, $\psi'\leq 0$ and $0\leq\psi\leq 1$. Since
$$
\nabla^2 \psi(r)=\psi'\nabla^2 r+\psi'' dr\otimes dr,
$$
by standard Hessian comparison (see e.g. \cite{Pet}) we known that, outside the cutlocus of $p$, one has
$$
\nabla^2 r \leq \frac{n-1}{r}g_{ij},
$$
and thus the function $\psi(r)$ satisfies $\psi'\leq 0$ and
$$
|\nabla \psi(r)|\leq C,\qquad   |\nabla f|^2 \Delta\psi+(p-2)\nabla^2\psi(\nabla f,\nabla f)\geq -\frac{C}{r}|\nabla f|^2.
$$
Therefore, the function $\phi(r)=\psi\left(\frac{r}{R}\right)^{1/\delta}$ satisfies \eqref{eq-ap} (outside the cutlocus of $p$). To overcome the lack of regularity in the cutlocus of $p$ one can use the so called Calabi trick. Therefore all the arguments in the proof Proposition \ref{p-ge} go through  and we obtain the following extension to the Riemannian setting:
\begin{proposition}\label{p-ge-r}  Let $(M^n,g)$ be a complete Riemannian manifold with nonnegative sectional curvature. Let $u$ be a positive weak solution of equation \eqref{eq-pl} with $1<p<n$. Then, for every $0<\eps<\frac{p-1}{n-p}$ it holds
$$
|\nabla u|\leq C\left(\sup_{B_{2R}(x_0)} u^{\frac{1}{n-p}+\eps}+R^{-\eps\frac{n-p}{p-1}}\right) u^{\frac{n-1}{n-p}-\eps}\qquad\text{on }B_{R}(x_0)
$$
for some $C=C(n,p,\eps)>0$, for every $R>0$ and every $x_0\in M^n$.
\end{proposition}
We explicitly note that this estimate is used only in the case $2<p<n$.
\begin{remark}
Gradient estimates for positive $p-$harmonic functions have been obtained in \cite{kotni} and \cite{wanzha}. In particular in \cite{wanzha} the authors managed to avoid imposing conditions on the sectional curvature, using integral estimates based on a Moser iteration argument, which only uses nonnegative Ricci curvature and first derivatives of the distance function. We expect that the same argument could work in our setting.
\end{remark}

\item In all the proofs of  Theorems \ref{t-e1}-\ref{t-se0}-\ref{t-se1}-\ref{t-se2} we used the fact that volume of geodesic balls has at most Euclidean growth, which is guaranteed by our curvature assumptions (i) and (ii), as already observed.

\item The final step in the proofs of Theorems \ref{t-e1}-\ref{t-se0}-\ref{t-se1}-\ref{t-se2}, in the Riemannian setting, goes as follows. From \eqref{eq-ap2}, we obtain
$$
\int_M u^{\frac{(n-1)p}{n-p}}|\mathring{\mathbf{V}}|^2+ \int_M u^{\frac{(n-1)p}{n-p}}\mathrm{Ric}(\mathbf{v},\mathbf{v})=0
$$
i.e.
\begin{equation}\label{eq12r}
\mathring{\mathbf{V}} = \nabla  \mathbf{v}-\frac{\diver\mathbf{v}}{n}g \equiv 0\quad\text{in } \Omega_{cr}^c,\qquad \mathrm{Ric}(\mathbf{v},\mathbf{v})\equiv 0\quad\text{in }M.
\end{equation}
Let $\Omega_0\subseteq\Omega_{cr}^c $ be a connected component of $\Omega_{cr}^c$. Arguing as in the proof of Theorem \ref{t-e1}, by elliptic regularity, we have $\diver\mathbf{v}\in C^{1,\alpha}_{\text{loc}}(\Omega_0)$. Differentiating the first identity in \eqref{eq12r}, we get
$$
\nabla_i \diver \mathbf v = n \nabla_j\nabla_i \mathbf v_j = n \nabla_i \diver \mathbf v - \mathrm{Ric}_{ij}\mathbf v_j=n \nabla_i \diver \mathbf v.
$$
Therefore $\diver\mathbf v=\text{const}$ on $\Omega_0$. Hence, the vector field $\mathbf v$ is homotetic, i.e. it satisfies
$$
\nabla_i \mathbf{v}_j + \nabla_j \mathbf{v}_i = \lambda g_{ij}, \quad\lambda\in\RR.
$$
Therefore, by a classical result of Kobayashi  \cite{kob}, we have that $\Omega_0$ must be locally Euclidean and we conclude as in the proof of Theorem \ref{t-e1}.
\end{itemize}

\

\

\begin{ackn}
\noindent The authors would like to thank Prof. A. Farina for helpful discussions concerning the regularity of solutions which greatly improved the exposition in Section \ref{Pre}.

\noindent The first author is member of the {\em GNSAGA, Gruppo Nazionale per le Strutture Algebriche, Geometriche e le loro Applicazioni} of INdAM. The second and the third authors are  members of {\em GNAMPA, Gruppo Nazionale per l'Analisi Matematica, la Probabilit\`a e le loro Applicazioni} of INdAM.
\end{ackn}

\

\

\noindent{\bf Data availability statement}

\noindent Data sharing not applicable to this article as no datasets were generated or analysed during the current study.

\

\

\

\

\end{document}